\documentclass[11pt,b5paper,notitlepage]{article}
\usepackage[b5paper, margin={0.5in,0.65in}]{geometry}
\usepackage{amsmath,amscd,amssymb,amsthm,mathrsfs,amsfonts,layout,indentfirst,graphicx,caption,mathabx, stmaryrd,appendix,calc,imakeidx,upgreek} % mathabx for \wtidecheck
\usepackage[dvipsnames]{xcolor}
\usepackage{palatino}  %template

\usepackage{slashed} % Dirac operator
\usepackage{mathrsfs} % Enable using \mathscr
\usepackage{extarrows} % long equal sign, \xlongequal{blablabla}
\usepackage{enumitem} % enumerate label change e.g. [label=(\alph*)]  shows (a) (b) 

\usepackage{CJK}   % Chinese package

\usepackage{csquotes} % \begin{displayquote}   \begin{displaycquote}  for quotation
\usepackage{epigraph}   %\epigraph{}{}  for quotation
\usepackage{fancyhdr} % date in footer

\usepackage{relsize} % use \mathlarger \larger \text{\larger[2]$...$} to enlarge the size of math symbols

\usepackage{verbatim}  % comment environment

\usepackage{halloweenmath} % Interesting halloween math symbols

\usepackage{tcolorbox}
\tcbuselibrary{theorems}
\usepackage{pdfrender}

\usepackage{tikz}
\usetikzlibrary{fadings}
\usetikzlibrary{patterns}
\usetikzlibrary{shadows.blur}
\usetikzlibrary{shapes}

\usepackage{tikz-cd}
\usepackage[nottoc]{tocbibind}   % Add  reference to ToC

\makeindex

\usepackage{lipsum}  
\let\OLDthebibliography\thebibliography
\renewcommand\thebibliography[1]{
	\OLDthebibliography{#1}
	\setlength{\parskip}{0pt}
	\setlength{\itemsep}{2pt} 
}

\allowdisplaybreaks  %allow aligns to break between pages
\usepackage{latexsym}
\usepackage{chngcntr}
\usepackage[colorlinks,linkcolor=blue,anchorcolor=blue, linktocpage,
]{hyperref}
\hypersetup{ urlcolor=cyan,
	citecolor=[rgb]{0,0.5,0}}

\counterwithin{figure}{section}

\theoremstyle{definition}
\newtheorem{df}{Definition}[section]
\newtheorem{eg}[df]{Example}

\newtheorem{rem}[df]{Remark}

\theoremstyle{plain}
\newtheorem{thm}[df]{Theorem}

\newtheorem{pp}[df]{Proposition}
\newtheorem{co}[df]{Corollary}
\newtheorem{lm}[df]{Lemma}

\usepackage{calligra}
\DeclareMathOperator{\shom}{\mathscr{H}\text{\kern -3pt {\calligra\large om}}\,}
\DeclareMathOperator{\sext}{\mathscr{E}\text{\kern -3pt {\calligra\large xt}}\,}
\DeclareMathOperator{\Rel}{\mathscr{R}\text{\kern -3pt {\calligra\large el}~}\,}
\DeclareMathOperator{\sann}{\mathscr{A}\text{\kern -3pt {\calligra\large nn}}\,}
\DeclareMathOperator{\send}{\mathscr{E}\text{\kern -3pt {\calligra\large nd}}\,}
\DeclareMathOperator{\stor}{\mathscr{T}\text{\kern -3pt {\calligra\large or}}\,}

\usepackage{aurical}
\DeclareMathOperator{\VVir}{\text{\Fontlukas V}\text{\kern -0pt {\Fontlukas\large ir}}\,}

\usepackage{aurical}
\usepackage[T1]{fontenc}

\newcommand{\fk}{\mathfrak}
\newcommand{\mc}{\mathcal}
\newcommand{\wtd}{\widetilde}
\newcommand{\wht}{\widehat}
\newcommand{\wch}{\widecheck}
\newcommand{\ovl}{\overline}

\newcommand{\Tr}{\mathrm{Tr}}
\newcommand{\End}{\mathrm{End}} %endomorphism

\newcommand{\id}{\mathrm{id}}
\newcommand{\Hom}{\mathrm{Hom}}

\newcommand{\opp}{\mathrm{op}}

\newcommand{\scr}{\mathscr}

\newcommand{\blt}{\bullet}
\newcommand{\Vbb}{\mathbb V}

\newcommand{\Mbb}{\mathbb M}
\newcommand{\Gbb}{\mathbb G}
\newcommand{\Cbb}{\mathbb C}
\newcommand{\Nbb}{\mathbb N}
\newcommand{\Zbb}{\mathbb Z}

\newcommand{\Ebb}{\mathbb E}

\newcommand{\Ker}{\mathrm{Ker}}

\newcommand{\Mod}{\mathrm{Mod}}
\newcommand{\ModR}{\mathrm{Mod}^{\mathrm R}}

\newcommand{\dps}{\displaystyle}

\newcommand{\eps}{\varepsilon}

\newcommand{\QC}{\mathrm{QCoh}_{\mathrm L}}
\newcommand{\Coh}{\mathrm{Coh}_{\mathrm L}}
\newcommand{\rad}{\mathrm{rad}}
\newcommand{\SLF}{\mathrm{SLF}}
\newcommand{\trc}{\mathrm{tr}}

\usepackage{tipa} % wierd symboles e.g. \textturnh

\usepackage{tipx}

\numberwithin{equation}{section}

\title{Pseudotraces on Almost Unital and Finite-Dimensional Algebras}
\author{{\sc Bin Gui, Hao Zhang}
}
\date{}

\begin{document}\sloppy % avoid stretch into margins
	\pagenumbering{arabic}
	%\pagenumbering{gobble}
	\setcounter{page}{1}
	\setcounter{section}{-1}
	%\setcounter{equation}{6}

	%%%%%%%%%%%%%%%%%%%%%%%%%%%%%%%%%%%%%%%%%%%%%%%%%%%%%%%%%

	\maketitle
%\small   \hyperlink{page.7}{Last page of TOC}

%\hyperlink{beforeindex}{Last page before index}~~~~~~  
%\hypertarget{beforeindex}{}

%%%%%%%%%%%%%%%%%%%%%%%%%%%%%
%\vspace{-0.5cm}

%\makeatletter
%\newcommand*{\toccontents}{\@starttoc{toc}}
%\makeatother
%\toccontents

% title and table of contents same page, no content title

%%%%%%%%%%%%%%%%%%%%%%%%%%%%%

\normalsize

%\hyperlink{beforeindex}{Current page of writing}~~~~~~ 
%\hypertarget{beforeindex}{}

\begin{abstract}
We introduce the notion of almost unital and finite-dimensional (AUF) algebras, which are associative $\Cbb$-algebras that may be non-unital or infinite-dimensional, but have sufficiently many idempotents. We show that the pseudotrace construction, originally introduced by Hattori and Stallings for unital finite-dimensional algebras, can be generalized to AUF algebras. 

Let $A$ be an AUF algebra. Suppose that $G$ is a projective generator in the category $\Coh(A)$ of finitely generated left $A$-modules that are quotients of free left $A$-modules, and let $B = \End_{A,-}(G)^\opp$. We prove that the pseudotrace construction yields an isomorphism between the spaces of symmetric linear functionals $\SLF(A)\xlongrightarrow{\simeq} \SLF(B)$, and that the non-degeneracies on the two sides are equivalent.
\end{abstract}

\tableofcontents

\section{Introduction}

In \cite{Miy-modular-invariance}, Miyamoto introduced the pseudo-$q$-trace construction for modules of vertex operator algebras (VOAs), generalizing the usual $q$-trace. His primary motivation was to address the failure of modular invariance for $q$-traces in the case of $C_2$-cofinite but irrational VOAs. While Zhu's theorem in \cite{Zhu-modular-invariance} establishes modular invariance for $q$-traces in the rational setting, this result does not extend to the irrational case---unless $q$-traces are replaced with pseudo-$q$-traces.

Miyamoto's original approach is quite involved. Moreover, his dimension formula for the space of torus conformal blocks is expressed in terms of higher Zhu algebras. This presents two drawbacks: first, higher Zhu algebras are difficult to compute in practice; second, their connection to the VOA module category is not transparent.

Later, Arike \cite{Ari10} and Arike-Nagatomo \cite{AN-pseudo-trace} introduced a simplified version of the pseudo-$q$-trace construction based on the idea of Hattori \cite{Hattori65} and Stallings \cite{Stallings}. Below, we briefly outline this approach.

Let $A$ be an algebra, and let $B$ be a unital finite-dimensional algebra. Let $M$ be a finite-dimensional $A$-$B$ bimodule, projective as a right $B$-module. By the projectivity, there is a (finite) left coordinate system of $M$, namely, elements $\alpha_1,\dots,\alpha_n\in\Hom_B(B,M)$ and $\wch\alpha^1,\dots,\wch\alpha^n\in\Hom_B(M,B)$ satisfying $\sum_i\alpha_i\circ\wch\alpha^i=\id_M$. Then the linear map 
\begin{gather*}
A\rightarrow B\qquad x\mapsto\sum_i \wch\alpha^i\circ x\circ\alpha_i(1_B)
\end{gather*}
descends to a linear map $A/[A,A]\rightarrow B/[B,B]$ which is independent of the choice of the left coordinate system. Its pullback gives a linear map
\begin{align}\label{eq18}
\SLF(B)\rightarrow\SLF(A)\qquad \phi\mapsto\Tr^\phi
\end{align}
where $\SLF(A)$ is the space of symmetric linear functionals on $A$---that is, linear maps $\psi:A\rightarrow\Cbb$ satisfying $\psi(xy)=\psi(yx)$ for all $x,y\in A$---and $\SLF(B)$ is the space of symmetric linear functionals on $B$. The above map is called the \textbf{pseudotrace construction}. Note that a typical choice of $A$ is $\End_B(M)$.

The pseudotrace construction is applied to the VOA setting as follows. Let $\Vbb$ be an $\Nbb$-graded $C_2$-cofinite VOA with central charge $c$, and let $\Mbb$ be a grading-restricted generalized $\Vbb$-module. Then $\Mbb$ admits a decomposition $\Mbb=\bigoplus_{\lambda\in\Cbb}\Mbb_{[\lambda]}$ into generalized eigenspaces of $L(0)$, where each $\Mbb_{[\lambda]}$ is finite-dimensional. Let $\End_\Vbb(\Mbb)$ be the algebra of linear operators on $\Mbb$ commuting with the action of $\Vbb$, which is necessarily unital and finite-dimensional. Let $B$ be a unital subalgebra of $\End_\Vbb(\Mbb)^\opp$. Assume that $\Mbb$ is a projective right $B$-module, equivalently, each $\Mbb_{[\lambda]}$ is $B$-projective. Let $\phi\in\SLF(B)$. Then for $v\in\Vbb$, the expression
\begin{align}\label{eq17}
\Tr^\phi(Y_\Mbb(v,z)q^{L(0)-\frac c{24}})=\sum_{\lambda\in\Cbb}\Tr^\phi\big(P(\lambda)Y_\Mbb(v,z)q^{L(0)-\frac c{24}}P(\lambda)\big)
\end{align}
converges absolutely for $z\in\Cbb$ and $0<|q|<1$, and defines a torus conformal block. Here, $P(\lambda)$ is the projection of $\ovl\Mbb:=\prod_{\mu\in\Cbb}\Mbb_{[\mu]}$ onto $\Mbb_{[\mu]}$. Then each $P(\lambda)Y_\Mbb(v,z)q^{L(0)-\frac c{24}}P(\lambda)$ is a linear operator on $\Mbb_{[\lambda]}$ commuting with the right action of $B$, and hence $\Tr^\phi$ can be defined on it.

Based on this formulation, in \cite[Conjecture 5.8]{GR-Verlinde}, Gainutdinov and Runkel proposed a conjecture that directly relates the space of torus conformal blocks of a $C_2$-cofinite VOA $\Vbb$ to the linear structure of the category $\Mod(\Vbb)$ of grading-restricted generalized $\Vbb$-modules. Let $\Gbb$ be a projective generator in $\Mod(\Vbb)$, and let $B=\End_\Vbb(\Gbb)$. Then $\Gbb$ is $B$-projective. The conjecture asserts that the linear map sending each $\phi\in\SLF(B)$ to \eqref{eq17} defines an isomorphism between $\SLF(B)$ and the space of torus conformal blocks of $\Vbb$.

The purpose of this note is to establish results in the theory of associative algebras that are essential for proving the Gainutdinov-Runkel conjecture. The actual resolution of the conjecture will appear in the forthcoming paper \cite{GZ5}.

Our approach stems from recognizing a structural analogy between the Gainutdinov-Runkel conjecture and a classical result in associative algebra: If $A$ is a unital finite-dimensional algebra and $M$ is a projective generator in the category of finite-dimensional left $A$-modules, then $M$ is projective over $B:=\End_A(M)^\opp$, and the pseudotrace map \eqref{eq18} is a linear isomorphism. This result was suggested in \cite[Sec. 2]{BBG-modified-trace} and was proved in \cite{Ari10} in the special case that $M=Ae$ where $e$ is a basic idempotent.

However, this classical result is not directly applicable to the Gainutdinov-Runkel conjecture. We need to generalize it to a larger class of associative algebras than unital finite-dimensional ones. In particular, we must consider infinite-dimensional algebras that can be approximated, in a certain sense, by finite-dimensional (and possibly unital) algebras. The need to consider infinite-dimensional associative algebras in the study of irrational VOAs has also been recognized in recent years from different perspectives, such as Huang's associative algebra $A^\infty(\Vbb)$ introduced in \cite{Hua-associative}, and the mode transition algebra introduced by Damiolini-Gibney-Krashen in \cite{DGK2}.

The infinite-dimensional algebra required for the proof of the Gainutdinov-Runkel conjecture is different from the above mentioned algebras. In \cite{GZ5}, we will show that the end
\begin{align*}
\dps\Ebb:=\int_{\Mbb\in\Mod(\Vbb)}\Mbb\otimes_\Cbb\Mbb'
\end{align*}
a priori an object of $\Mod(\Vbb^{\otimes2})$, carries a structure of an associative $\Cbb$-algebra that is compatible with its $\Vbb^{\otimes2}$-module structure. This algebra $\Ebb$ is an example of an \textbf{almost unital and finite-dimensional algebra}\footnote{Here, ``almost" modifies the entire phrase ``unital and finite-dimensional", not just ``unital".} (abbreviated as \textbf{AUF algebra}), meaning that $\Ebb$ has a collection of mutually orthogonal idempotents $(e_i)_{i\in\fk I}$ such that $\Ebb=\sum_{i,j\in\fk I}e_i\Ebb e_j$ where each summand $e_i\Ebb e_j$ is finite-dimensional. (This sum is automatically direct.) In fact, $\Ebb$ has only finitely many irreducibles. We call such an algebra \textbf{strongly AUF}.

The main result of this note is a generalization of the aforementioned isomorphism between spaces of symmetric linear functionals to the setting of strongly AUF algebras. More precisely, we prove that the pseudotrace construction defines a linear isomorphism $\SLF(B)\simeq\SLF(A)$ where $A$ is strongly AUF, $M$ is a projective generator of the category $\Coh(A)$ of \textbf{coherent left $A$-modules} (i.e., finitely generated left $A$-modules that are quotients of free ones), and $B=\End_A(M)^\opp$. See Thm. \ref{lb44}. Moreover, we show that the symmetric linear functional on $B$ is non-degenerate if and only if the corresponding functional on $A$ is non-degenerate. See Thm. \ref{lb64}.

Since the associative algebra structure on the end $\Ebb$ will not be developed in this note, we present some alternative examples of AUF algebras for illustration. Let $U(\Vbb)$ be the universal algebra of $\Vbb$ as defined in \cite{FZ92-universal-alg}. Let
\begin{align*}
U(\Vbb)^{\mathrm{reg}}=\bigoplus_{\lambda,\mu\in\Cbb}U(\Vbb)_{[\lambda,\mu]}
\end{align*}
where $U(\Vbb)_{[\lambda,\mu]}$ is the subspace of joint generalized-eigenvectors of the left and right actions of $L(0)$ corresponding to the eigenvalues $\lambda$ and $\mu$ respectively. The following properties are shown in \cite{MNT10}: Each $U(\Vbb)_{[\lambda,\mu]}$ is finite-dimensional. For each $\lambda,\mu,\nu\in\Cbb$ one has
\begin{align*}
U(\Vbb)_{[\lambda,\mu]}U(\Vbb)_{[\mu,\nu]}\subset U(\Vbb)_{[\lambda,\nu]}
\end{align*}
In particular, $U(\Vbb)^{\mathrm{reg}}$ is a subalgebra of $U(\Vbb)$. Moreover, there is an increasing sequence of idempotents $(1_n)_{n\in\Zbb_+}$ such that $U(\Vbb)^{\mathrm{reg}}=\bigcup_n1_nU(\Vbb)^{\mathrm{reg}}1_n$. (See \cite[Sec. 2.6]{MNT10}.) Therefore, $U(\Vbb)^{\mathrm{reg}}$ is AUF, since the family of orthogonal idempotents in the definition of AUF algebras can be chosen to be $(1_{n+1}-1_n)_{n\in\Zbb_+}$.

For a more elementary and concrete example, consider the following. Let $B$ be a unital finite-dimensional algebra. Let $M$ be a right $B$-module. Equip $M$ with a grading
\begin{align*}
M=\bigoplus_{i\in\fk I}M(i)
\end{align*}
where each $M(i)$ is finite-dimensional and is preserved by the right action of $B$. Let $A$ be
\begin{align*}
\End^0_B(M):=\{T\in\End(M):&(Tm)b=T(mb)\text{ for all $m\in M,b\in B$,}\\
&T|_{M(i)}=0\text{ for all but finitely many }i\in\fk I \}
\end{align*}
Then $A$ is clearly an AUF algebra, with the family of mutually orthogonal idempotents given by the projections $e_i$ of $M$ onto $M(i)$. 

In fact, any strongly AUF algebra arises from such a construction. More precisely, an algebra is strongly AUF if and only if it is isomorphic to some $\End^0_B(M)$, where $M$ and $B$ satisfy the above conditions and, in addition, $M$ is a projective generator in the category of right $B$-modules. See Thm. \ref{lb65}.

Note that the relationship between $\End^0_B(M)$ and $C_2$-cofinite VOAs is straightforward: If $\Mbb\in\Mod(\Vbb)$ is equipped with the grading $\bigoplus_{\lambda\in\Cbb}\Mbb_{[\lambda]}$ given by the generalized eigenspaces of $L(0)$, and if $B$ is a unital subalgebra of $\End_\Vbb(\Mbb)^\opp$ such that $\Mbb$ is projective as a right $B$-module, then each $P(\lambda)Y_\Mbb(v,z)q^{L(0)-\frac c{24}}P(\lambda)$ appearing in \eqref{eq17} lies in $\End^0_B(\Mbb)$. Therefore, the main result of this note on pseudotraces (Thm. \ref{lb64}) can be applied to $C_2$-cofinite VOAs. Details of this application will be presented in \cite{GZ5}.

\subsection*{Acknowledgment}

We are grateful to the referee for the careful report and valuable suggestions. B.G. is supported by NSFC Grant 12401159.

\section{Preliminaries}

Throughout this note, algebras are associative, not necessarily unital, and over $\Cbb$. Let $\Nbb=\{0,1,2,\dots\}$ and $\Zbb_+=\{1,2,\dots\}$. For any vector spaces $V,W$, we let $\Hom(V,W)=\Hom_\Cbb(V,W)$ be the space of linear maps $V\rightarrow W$, and let $\End(V)=\Hom(V,V)$.

Let $A$ be an algebra. Its opposite algebra is denoted by $A^\opp$. If $M,N$ are left (resp. right) $A$-modules, we let $\Hom_{A,-}(M,N)$ (resp. $\Hom_{-,A}(M,N)$) be the space of linear maps $M\rightarrow N$ intertwining the left (resp. right) actions of $A$.

An \textbf{idempotent} $e\in A$ is an element satisfying $e^2=e$. If $e,f\in A$ are idempotents, we write $e\leq f$ if $ef=fe=e$. Equivalently, $f=e+e'$ where $e'\in A$ is an idempotent \textbf{orthogonal} to $e$ (i.e. $ee'=e'e=0$). We call $e$ a \textbf{sub-idempotent} of $f$ if $e\leq f$. We say that a nonzero idempotent $e$ is \textbf{primitive} if the only idempotent $f$ satisfying $f\leq e$ is $f=0$ and $f=e$.

In this section, we review some well-known facts about associative algebras. Since, unlike many references, our algebras are not assumed to be unital, we include proofs for the reader's convenience.

\begin{df}
Let $u,v\in A$. We say that $(u,v)$ is a pair of \textbf{partial isometries in $A$} if the following are true:
\begin{enumerate}[label=(\alph*)]
\item $p:=vu$ and $q:=uv$ are idempotents.
\item $u\in qAp$ and $v\in pAq$.
\end{enumerate}
In this case, we also say that $u$ is a partial isometry from $p$ to $q$, and that $v$ is a partial isometry from $q$ to $p$. We say that two idempotents are \textbf{equivalent} if there are partial isometries between them.
\end{df}

\begin{pp}\label{lb6}
Let $e,f\in A$ be idempotents. Then an element of $\Hom_{A,-}(Ae,Af)$ is precisely the right multiplication of an element of $eAf$. In particular, we have an algebra isomorphism
\begin{align*}
\End_{A,-}(Ae)^\opp\simeq eAe
\end{align*}
\end{pp}

%From the following proof, it is clear that this proposition is true even without assuming that $A$ is almost unital.

\begin{proof}
Clearly the right multiplication by some element of $eAf$ yields an element of $\Hom_{A,-}(Ae,Af)$. Conversely, suppose that $T\in \Hom_{A,-}(Ae,Af)$. Let $x=T(e)$, which belongs to $Af$. Since $ex=eT(e)=T(ee)=T(e)=x$, we see that $x\in eAf$. For each $y\in A$, we have $T(ye)=yT(e)=yx=yex$, which shows that $T$ is the right multiplication by $x$.
\end{proof}

\begin{co}\label{lb13}
Let $e,f$ be idempotents in $A$. The following are equivalent: 
\begin{enumerate}[label=(\arabic*)]
\item $Ae\simeq Af$ as left $A$-modules.
\item There is a partial isometry from $e$ to $f$.
\end{enumerate}
\end{co}

\begin{proof}
(1)$\Rightarrow$(2): Let $T\in\Hom_{A,-}(Ae,Af)$ be an isomorphism with inverse $T^{-1}\in\Hom_{A,-}(Af,Ae)$. By Prop. \ref{lb6}, $T$ and $T^{-1}$ are realized by the right multiplications of $u\in eAf$ and $v\in fAe$ respectively. Since $TT^{-1}=1_{Af}$, we have $vu=f$. Since $T^{-1}T=1_{Ae}$, we have $uv=e$.

(2)$\Rightarrow$(1): Let $u\in eAf$ and $v\in fAe$ such that $uv=e,vu=f$. Then the right multiplication of $u$ on $Ae$ has inverse being the right multiplication of $v$. So $Ae\simeq Af$.
\end{proof}

\begin{co}\label{lb3}
Let $e\in A$ be an idempotent. Let $M$ be a left $A$-submodule of $Ae$. The following are equivalent.
\begin{enumerate}
\item[(1)] $M$ is a direct summand of $Ae$.
\item[(2)]  $M=Af$ for some idempotent $f\leq e$ in $A$.
\end{enumerate}
\end{co}

\begin{proof}
(2)$\Rightarrow$(1): $Ae=Af\oplus Af'$ where $f'=e-f$ is an idempotent.

(1)$\Rightarrow$(2): Let $Ae=M\oplus N$. Let $\varphi:Ae\rightarrow Ae$ be the projection on $M$ vanishing on $N$. Then $\varphi\in\End(Ae)$. By Prop. \ref{lb6}, $\varphi$ is the right multiplication by some $f\in eAe$. Since $\varphi\circ\varphi=\varphi$, clearly $f^2=f$. Moreover, $M=\varphi(Ae)=(Ae)f=Af$. 
\end{proof}

\begin{co}\label{lb5}
Let $e\in A$ be an idempotent. The following are equivalent.
\begin{enumerate}
\item[(1)] $Ae$ is an indecomposable left $A$-module.
\item[(2)] $e$ is primitive. 
\end{enumerate}
\end{co}

\begin{proof}
This follows immediately from Cor. \ref{lb3}.
\end{proof}

\begin{lm}\label{lb7}
Let $M$ be a nonzero finitely-generated left $A$-module. Then $M$ has a maximal proper left $A$-submodule $N$. Consequently, there is an epimorphism of $M$ onto an irreducible module.
\end{lm}

%Note that $N$ is necessarily in $\QC(A)$, but not necessarily in $\Coh(A)$.

\begin{proof}
Let $\xi_1,\dots,\xi_n$ generate $M$. Without loss of generality, we assume that $\xi_1$ does not belong to the submodule $N_0$ generated by $\xi_2,\dots,\xi_n$. By Zorn's lemma, there is a left submodule $N\leq M$ maximal with respect to the property that $N_0\subset N$ and $\xi_1\notin N$. Let us prove that $N$ is a maximal proper submodule. Let $N<K\leq M$. Then by the maximality of $N$ we must have $\xi_1\in K$. So $\xi_1,\dots,\xi_n\in K$, and hence $K=M$. So $K$ is not proper.
\end{proof}

\section{Almost unital algebras}

In this section, we introduce the notion of almost unital algebras, which is weaker than being almost unital and finite-dimensional.

\begin{df}
We say that an algebra $A$ is \textbf{almost unital} if the following conditions are satisfied:
\begin{itemize}
\item[(a)] For each $x\in A$, there is an idempotent $e\in A$ such that $x=exe$.
\item[(b)] For any finitely many idempotents $e_1,\dots,e_n\in A$ there exists an idempotent $e\in A$ such that $e_i\leq e$ for all $1\leq i\leq n$.
\end{itemize}
\end{df}

Throughout this section, unless otherwise stated, $A$ is assumed to be almost unital.

\begin{df}\label{lb18}
We say that a left $A$-module $M$ is \textbf{quasicoherent} if one of the following equivalent conditions holds:
\begin{enumerate}[label=(\arabic*)]
\item For each $\xi\in M$ we have $\xi\in A\xi$.
\item For each $\xi\in M$ there exists an idempotent $e\in A$ such that $\xi=e\xi$.
\item $M$ is a quotient module of $\bigoplus_{i\in I}Ae_i$ where each $e_i\in A$ is an idempotent.
\item $M$ is a quotient module of a free left $A$-module $A^{\oplus I}$.
\end{enumerate}
The category of quasicoherent left $A$-modules is denoted by \pmb{$\QC(A)$}.
\end{df}

\begin{proof}[Proof of equivalence]
(1)$\Rightarrow$(2): For each $\xi\in M$, since $\xi\in A\xi$, we have $\xi=a\xi$ for some $a\in A$. Choose idempotent $e\in A$ such that $a\in eAe$. Then $e\xi=ea\xi=a\xi=\xi$.

(2)$\Rightarrow$(1): Obvious.

(2)$\Rightarrow$(3): For each $\xi\in M$, let $e_\xi\in A$ be an idempotent such that $e_\xi\xi=\xi$. Then we have a morphism $\bigoplus_{\xi\in M}Ae_\xi\rightarrow M$ whose restriction to $Ae_\xi$ sends each $a\in Ae_\xi$ to $a\xi$. Then $\xi=e_\xi\xi$ implies that $\xi\in Ae_\xi\xi$, and hence $\xi$ is in the range of this morphism. So this morphism is surjective.

(3)$\Rightarrow$(4): This is obvious, since we have an epimorphism $A\rightarrow Ae_i$ and hence an epimorphism $\bigoplus_{i\in I}A\rightarrow \bigoplus_{i\in I}Ae_i$.

(4)$\Rightarrow$(2): It suffices to show that $A^{\oplus I}$ satisfies the requirement of (2). Choose $\xi=(a_i)_{i\in I}\in A^{\oplus I}$. Then there are only finitely many $i\in I$ such that $a_i\ne 0$. Since $A$ is almost unital, there exist idempotents $e_i\in A$ (where $i\in I$) such that $a_i=e_i a_i e_i$ for all $i\in I$. (If $a_i=0$, then we choose $e_i=0$). Choose an idempotent $e\in A$ such that $e_i\leq e$ for all $i\in I$. Then $\xi=e\xi$.
\end{proof}

\begin{df}\label{lb26}
A left $A$-module $M$ is called \textbf{coherent} if it is quasicoherent and finitely-generated. By the above proof of equivalence, it is clear that $M$ is coherent iff $M$ is a quotient of $\bigoplus_{i\in I}Ae_i$ where $I$ is a \textit{finite} index set and $e_i\in A$ is an idempotent. The category of coherent left $A$-modules is denoted by \pmb{$\Coh(A)$}.
\end{df}

However, note that a coherent left $A$-module is not necessarily a quotient of $A^{\oplus n}$ where $n\in\Zbb_+$. Indeed, $A$ is not necessarily finitely generated as a left $A$-module.

\begin{rem}\label{lb55}
If $M\in\QC(A)$, then every submodule of $M$ is quasicoherent, and every quotient module of $M$ is quasicoherent. However, if $M\in \Coh(A)$, then a submodule of $M$ is not known to be coherent. Thus, $\QC(A)$ is an abelian category, while $\Coh(A)$ is not known to be abelian.
\end{rem}

\begin{pp}\label{lb1}
Let $M\in\QC(A)$. The following are equivalent.
\begin{enumerate}[label=(\arabic*)]
\item $M$ is projective in the category of left $A$-modules.
\item $M$ is projective in $\QC(A)$.
\item $M$ is a direct summand of $\bigoplus_{i\in I}Ae_i$ for some index set $I$ and each $e_i\in A$ is an idempotent. 
\end{enumerate}
\end{pp}

\begin{proof}
(3)$\Rightarrow$(1): It is well-known that a direct summand of a projective module is projective. Thus, it suffices to prove that $\bigoplus_{i\in I}Ae_i$ is projective.  Let $\Phi:\bigoplus_{i\in I}Ae_i\rightarrow N$ be a morphism where $N$ is a left $A$-module.  Let $\Gamma:K\rightarrow N$ be an epimorphism. Let
\begin{align*}
\eta_i=\Phi(e_i)
\end{align*}
Since $\Gamma$ is surjective, there is $\xi_i\in K$ such that $\Gamma(\xi_i)=\eta_i$. Define $\Psi:\bigoplus_{i\in I}Ae_i\rightarrow K$ to be the morphism sending each $ae_i\in Ae_i$ to $ae_i\xi_i$. Then the following commute:
\begin{equation*}
\begin{tikzcd}
            & \bigoplus_{i\in I}Ae_i \arrow[d,"\Phi"] \arrow[ld,"\Psi"'] \\
K \arrow[r,"\Gamma"'] & N                     
\end{tikzcd}
\qquad
\begin{tikzcd}
                     & ae_i \arrow[d, maps to] \arrow[ld, maps to] \\
ae_i\xi_i \arrow[r, maps to] &  ae_i\eta_i                                       
\end{tikzcd}
\end{equation*}
Note that $\mapsto$ holds since $\Gamma(ae_i\xi_i)=ae_i\Gamma(\xi_i)=ae_i\eta_i$, and \rotatebox[origin=c]{270}{\hbox{$\mapsto$}} holds since $\Phi(ae_i)=\Phi(ae_ie_i)=ae_i\Phi(e_i)=ae_i\eta_i$.

(1)$\Rightarrow$(2): Obvious.

(2)$\Rightarrow$(3): Choose an epimorphism $\bigoplus_{i\in I}Ae_i\rightarrow M$, which splits because $M$ is projective. So $M$ is a direct summand of $\bigoplus_{i\in I}Ae_i$.
\end{proof}

\begin{pp}\label{lb2}
Let $M\in\Coh(A)$. The following are equivalent.
\begin{enumerate}[label=(\arabic*)]
\item $M$ is projective in the category of left $A$-modules.
\item $M$ is projective in $\QC(A)$.
\item $M$ is projective in $\Coh(A)$.
\item $M$ is a direct summand of $\bigoplus_{i\in I}Ae_i$ for some \textit{finite} index set $I$ and each $e_i\in A$ is an idempotent.
\end{enumerate} 
\end{pp}

Therefore, there is no ambiguity when talking about projective coherent left $A$-modules.

\begin{proof}
Clearly we have (1)$\Rightarrow$(2) and (2)$\Rightarrow$(3). By Prop. \ref{lb1} we have (4)$\Rightarrow$(1). Assume (3). By Def. \ref{lb26}, there is an epimorphism $\bigoplus_{i\in I}Ae_i\rightarrow M$ such that $I$ is finite, and that it splits (because $M$ is projective in $\Coh(A)$). So (4) is true.
\end{proof}

\begin{rem}\label{lb10}
If $M\in\QC(A)$, clearly $M$ is irreducible in $\QC(A)$ iff $M$ is irreducible in the category of left $A$-modules; in this case we say that $M$ is \textbf{irreducible}. Note that even if $M\in\Coh(A)$, its irreducibility is understood as in $\QC(A)$ but not as in $\Coh(A)$.
\end{rem}

\begin{pp}\label{lb19}
Let $M$ be a left $A$-module. The following are equivalent.
\begin{enumerate}[label=(\arabic*)]
\item $M\in\QC(A)$ and $M$ is irreducible.
\item $M\simeq Ae/N$ where $e\in A$ is an idempotent and $N$ is a maximal proper left $A$-submodule of $Ae$.
\item $M\simeq A/N$ where $N$ is a maximal (proper) left ideal of $A$.
\end{enumerate}
\end{pp}

\begin{proof}
(1)$\Rightarrow$(2): Let $M\in\QC(A)$ be irreducible. By Def. \ref{lb18}, $M$ has an epimorphism $\Phi$ from some $\bigoplus_i Ae_i$ where $e_i\in A$ is an idempotent. The restriction of $\Phi$ to some $Ae_i$ must be nonzero, and hence must be surjective (since $M$ is irreducible). It follows that $M$ has an epimorphism $\Psi$ from $Ae_i$. Then $N=\Ker\Psi$ is a maximal proper left $A$-submodule of $Ae_i$, and $M\simeq Ae_i/N$.

(1)$\Rightarrow$(3): In the above proof, $M$ also has an epimorphism from $\bigoplus_i A$ (since we have an epimorphism $A\rightarrow Ae_i$). Thus, replacing $Ae_i$ with $A$ in the above proof, we are done.

(2),(3)$\Rightarrow$(1): Clearly $M$ is irreducible. That $M\in\QC(A)$ follows from Def. \ref{lb18}.
\end{proof}

\section{Projective covers}

Let $A$ be an algebra, not necessarily almost unital. In this section, we recall some basic facts about projective covers. When $A$ is unital, these results can be found in \cite{AF-rings}, for example. In the non-unital case, one can reduce to the unital setting by considering the unitalization of $A$. For the reader’s convenience, we include complete proofs.

\subsection{Basic facts}

\begin{df}
Let $M$ be a left $A$-module. A left $A$-submodule $K\leq M$ is called \textbf{superfluous}, if for any left $A$-submodule $L\leq M$ satisfying $K+L=M$ we must have $L=M$.
\end{df}

\begin{rem}\label{lb21}
Obviously, we have an equivalent description of superfluous submodules: Let $\pi:M\rightarrow M/K$ be the quotient map. Then $K\leq M$ is superfluous iff for any morphism of left $A$-modules $\varphi:N\rightarrow M$ such that $\pi\circ\varphi:N\rightarrow M/K$ is surjective, it must be true that $\varphi$ is surjective.
\end{rem}

\begin{df}
Let $M$ be a left $A$-module. A \textbf{projective cover} of $M$ denotes a left $A$-module epimorphism $\varphi:P\twoheadrightarrow M$ where $P$ is a projective left $A$-module, and $\Ker\varphi$ is superfluous in $P$.
\end{df}

The following property says that among the projective modules that have epimorphisms to $M$, the projective cover is the smallest one in the sense of direct summand.

\begin{pp}\label{lb23}
Let $\varphi:P\rightarrow M$ be a projective cover of $M$. Let $\psi:Q\rightarrow M$ be an epimorphism where $Q$ is projective. Then there is a morphism $\alpha:Q\rightarrow P$ such that the following diagram commutes. 
\begin{equation}\label{eq2}
\begin{tikzcd}
                       & P \arrow[d,"\varphi", two heads] \\
Q \arrow[ru,"\alpha"] \arrow[r,"\psi"', two heads] & M          
\end{tikzcd}
\end{equation}
Moreover, for any such $\alpha$, there is a left $A$-submodule $P'\leq Q$ such that $Q=\ker\alpha\oplus P'$ and that $\alpha|_{P'}:P'\xrightarrow{\simeq} P$ is an isomorphism.
\end{pp}

By setting $L=\ker\alpha$, it follows that \eqref{eq2} is equivalent to
\begin{equation}\label{eq3}
\begin{tikzcd}
                       & P \arrow[d,"\varphi", two heads] \\
L\oplus P \arrow[ru,"0\oplus \id_P"] \arrow[r,"0\oplus\varphi"', two heads] & M          
\end{tikzcd}
\end{equation}

\begin{proof}
The existence of $\alpha$ follows from the fact that $Q$ is projective and that $\varphi$ is an epimorphism. Moreover, since $\ker\varphi$ is superfluous and $\varphi\circ\alpha$ is surjective, by Rem. \ref{lb21}, $\alpha$ is surjective. Therefore, since $P$ is projective, the epimorphism $\alpha$ splits, i.e., there is a morphism $\beta:P\rightarrow Q$ such that $\alpha\circ\beta:P\rightarrow P$ equals $\id_P$. One sees that $P'=\beta(P)$ fulfills the requirement.
\end{proof}

It follows that projective covers are unique up to isomorphisms:

\begin{co}\label{lb22}
Let $M$ be a left $A$-module with projective covers $\varphi:P\rightarrow M$ and $\psi:Q\rightarrow M$. Then there exists an isomorphism $\alpha:Q\rightarrow P$ such that \eqref{eq2} commutes.
\end{co}

\begin{proof}
By Prop. \ref{lb23}, there exists $\alpha$ such that \eqref{eq2} commutes. It remains to show that $\alpha$ is an isomorphism. We assume that \eqref{eq2} equals \eqref{eq3}. Since $0\oplus\varphi:L\oplus P\rightarrow M$ is a projective cover, $L+\ker(\varphi)=\ker(0\oplus\varphi)$ is superfluous, and hence $L$ is superfluous. Thus, since $L+P$ equals $Q=L\oplus P$, we must have $Q=P$ and hence $L=0$. So $\alpha=0\oplus\id_P$ is an isomorphism.  
\end{proof}

\subsection{Projective covers of irreducibles}

\begin{pp}\label{lb24}
Suppose that $\varphi:P\rightarrow M$ is a projective cover of an irreducible left $A$-module $M$. Then $P$ is indecomposable.
\end{pp}

\begin{proof}
Suppose that $P=P'\oplus P''$. Then one of $\varphi|_{P'},\varphi|_{P''}$ (say $\varphi|_{P'}$) is nonzero. Since $M$ is irreducible,  $\varphi|_{P'}:P'\rightarrow M$ must be surjective. So the map $P'\hookrightarrow P'\oplus P''\xrightarrow{\varphi}M$ is surjective. Since $\ker\varphi$ is superfluous, by Rem. \ref{lb21}, $P'\hookrightarrow P'\oplus P''$ is surjective, and hence $P''=0$.
\end{proof}

\begin{thm}\label{lb8}
Let $e\in A$ be a primitive idempotent satisfying
\begin{align*}
\dim eAe<+\infty
\end{align*}
Let $K$ be any proper left $A$-submodule of $Ae$. Then $K$ is superfluous. In other words, the quotient map $Ae\rightarrow Ae/K$ is the projective cover of $Ae/K$.
\end{thm}

\begin{proof}
Step 1. Let $\varphi:Ae\rightarrow Ae/K$ be the quotient map. Let $N$ be a submodule of $Ae$. Assume that $N+K=Ae$; in other words, if we let $\iota:N\hookrightarrow Ae$ be the inclusion, then $\varphi\circ\iota:N\rightarrow Ae/K$ is surjective. Our goal is to show that $N=Ae$.

Since $Ae$ is projective and $\varphi\circ\iota$ is surjective, there is a morphism $\alpha:Ae\rightarrow N$ such that $\varphi=\varphi\circ\iota\circ\alpha$. Let $\beta=\iota\circ\alpha$. Then the following diagram commutes:
\begin{equation}\label{eq1}
\begin{tikzcd}[row sep=large]
                  &             & Ae \arrow[lld,"\alpha"'] \arrow[ld,"\beta"] \arrow[d, two heads, "\varphi"] \\
N \arrow[r, hook,"\iota"'] & Ae \arrow[r,"\varphi"', two heads] & Ae/K                               
\end{tikzcd}
\end{equation}
To prove that $\iota$ is surjective, it suffices to show that $\beta$ is surjective.\\[-1ex]

Step 2. Suppose that $\beta$ is not surjective. Let us find a contradiction. Since $\beta\in\End_{A,-}(Ae)$, by Prop. \ref{lb6}, $\beta$ is the right multiplication by some $x\in eAe$. Let $R_x:eAe\rightarrow eAe$ be the right multiplication of $x$ on $eAe$. Then $R_x$ is not surjective. Otherwise, there exists $a\in A$ such that $R_x(eae)=e$, i.e., $eaex=e$. Then for each $b\in A$, we have $be=beaex=\beta(beae)$, contradicting the fact that $\beta$ is not surjective.

It is well-known that if $T$ is a linear operator on a finite-dimensional $\Cbb$-vector space $W$, then $W$ is the direct sum of generalized eigenspaces of $T$, and the projection operator of $W$ onto each generalized eigenspace is a polynomial of $T$. Therefore, $R_x$ has only one eigenvalue. Otherwise, there is a polynomial $p$ such that $p(R_x)=R_{p(x)}$ is the projection of $eAe$ onto a proper subspace, and hence $p(x)$ is an idempotent in $eAe$ not equal to $0$ or $e$. This is impossible, since $e$ is assumed to be primitive.

Therefore, $R_x$ has a unique eigenvalue, which must be $0$ since $R_x$ is not surjective. By linear algebra, $R_x$ is nilpotent. Since $R_{x^n}=(R_x)^n$, it follows that $x$ is nilpotent, and hence $\beta$ is nilpotent. By \eqref{eq1}, we have $\varphi=\varphi\circ\beta$, and hence $\varphi=\varphi\circ\beta=\varphi\circ\beta^2=\varphi\circ\beta^3=\cdots=0$. This contradicts the fact that $\varphi$ is a surjection onto a nonzero module, finishing the proof.
\end{proof}

\begin{co}\label{lb11}
Let $e\in A$ be a primitive idempotent satisfying $\dim eAe<+\infty$. Then $Ae$ has a unique proper maximal left $A$-submodule, denoted by \pmb{$\rad(Ae)$}.
\end{co}

It follows from Thm. \ref{lb8} that $Ae$ is the projective cover of the irreducible $Ae/\rad(Ae)$.

\begin{proof}
By Lem. \ref{lb7}, $Ae$ has at least one proper maximal left $A$-submodule. Suppose that $K\neq L$ are proper maximal left $A$-submodules of $Ae$. By the maximality, we have $K+L=Ae$. By Thm. \ref{lb8}, $L$ is superfluous. So $K=Ae$, impossible. 
\end{proof}

\section{Left pseudotraces}

Let $A,B$ be algebras such that $B$ is unital. Fix an $A$-$B$ bimodule $M$. We do not assume that $M_B$ is unital, i.e., $1_B\in B$ acts as the identity on $M$.

\begin{df}\label{lb41}
A \textbf{left coordinate system} of $M$ denotes a collection of morphisms 
\begin{align}\label{eq10}
\alpha_i\in\Hom_{-,B}(B,M)\qquad \wch\alpha^i\in \Hom_{-,B}(M,B)
\end{align}
where $i$ runs through an index set $I$ such that the following conditions hold:
\begin{enumerate}[label=(\alph*)]
\item  For each $\xi\in M$, we have $\wch\alpha^i(\xi)=0$ for all but finitely many $i\in I$, and  $\sum_{i\in I}\alpha_i\circ\wch\alpha^i(\xi)=\xi$.
\item For each $x\in A$ (viewed as an element of $\End_{-,B}(M)$), we have $x\circ\alpha_i=0$ and $\wch\alpha^i\circ x=0$ for all but finitely many $i\in I$.
\end{enumerate}
\end{df}

\begin{rem}\label{lb20}
$M$ is a projective right $B$-module iff there exists $(\alpha_i,\wch\alpha^i)_{i\in I}$ of the form \eqref{eq10} satisfying condition (a).
\end{rem}
\begin{proof}
	Suppose that there exists $(\alpha_i,\wch\alpha^i)_{i\in I}$ such that (a) holds. Define morphisms of right $B$-modules
\begin{gather*}
\Phi:B^{\oplus I}\rightarrow M\qquad \oplus_i b_i\mapsto \sum_i \alpha_i(b_i)\\
\Psi:M\rightarrow B^{\oplus I}\qquad \xi\mapsto \oplus_i\wch\alpha^i(\xi)
\end{gather*}
Then (a) implies that $\Phi\circ \Psi=\id_M$. Thus, $M$ is a direct summand of $B^{\oplus I}$, and hence is projective as a right $B$-module.

Conversely, assume $M$ is projective as a right $B$-module. Then we have an epimorphism $\Phi:B^{\oplus I}\rightarrow M$ and a morphism $\Psi:M\rightarrow  B^{\oplus I}$ such that $\Phi\circ \Psi=\id_M$. For each $i\in I$, let $\iota_{i}:B\rightarrow B^{\oplus I}$ be the inclusion map of $B$ into the $i$-th direct summand, and $\pi_i:B^{\oplus I}\rightarrow B$ be the projection map onto the $i$-th direct summand. Set 
	\begin{align*}
		\alpha_i=\Phi \circ \iota_i\qquad \wch \alpha^i=\pi_i\circ \Psi
	\end{align*}
Then $(\alpha_i,\wch\alpha^i)_{i\in I}$ satisfies (a).
\end{proof}

\begin{df}
Assume that $M$ has a left coordinate system $(\alpha_i,\wch\alpha^i)_{i\in I}$. Define the  \textbf{\pmb{$B$}-trace function}
\begin{gather*}
\Tr^B:A\rightarrow B/[B,B]\qquad x\mapsto \sum_{i\in I}\wch\alpha^i\circ x\circ\alpha_i
\end{gather*}
where the RHS, originally an element of $\End_{-,B}(B)\simeq B$,\footnote{This isomorphism relies on the fact that $B$ is unital.} is descended to $B/[B,B]$.
\end{df}

\begin{lm}\label{lb9}
The definition of $\Tr^B$ is independent of the choice of left coordinate systems.
\end{lm}

\begin{proof}
Suppose that $(\beta_j,\wch\beta^j)_{j\in J}$ is another left coordinate system of the $A$-$B$ bimodule $M$. Let $I_x\subset I$ and $J_x\subset J$ be finite sets such that $\wch\alpha^i\circ x=0,x\circ\alpha_i=0$ for any $i\in I\setminus I_x$, and that $\wch\beta^j\circ x=0,x\circ\beta_j=0$ for any $j\in J\setminus J_x$. Then
\begin{align*}
\sum_{i\in I_x}\wch\alpha^i\circ x\circ\alpha_i=\sum_{i\in I_x,j\in J}\wch\alpha^i\circ x\circ \beta_j\circ\wch\beta^j\circ\alpha_i=\sum_{i\in I_x,j\in J_x}\wch\alpha^i\circ x\circ \beta_j\circ\wch\beta^j\circ\alpha_i
\end{align*}
Since each $\wch\alpha^i\circ x\circ \beta_j$ and $\wch\beta^j\circ\alpha_i$ are in $\End_{-,B}(B)\simeq B$, the RHS above equals
\begin{align*}
\sum_{i\in I_x,j\in J_x}\wch\beta^j\circ\alpha_i\circ\wch\alpha^i\circ x\circ \beta_j=\sum_{j\in J_x}\wch\beta^j\circ x\circ \beta_j
\end{align*}
in $B/[B,B]$.
\end{proof}

\begin{pp}\label{lb37}
$\Tr^B$ is \textbf{symmetric}, i.e., $\Tr^B(xy)=\Tr^B(yx)$ for any $x,y\in A$. Therefore, $\Tr^B$ descends to a linear map $A/[A,A]\rightarrow B/[B,B]$.
\end{pp}

\begin{proof}
Let $x,y\in A$. Let $I_0\subset I$ be a finite set such that $\wch\alpha^i\circ x=\wch\alpha^i\circ y=0$ and $x\circ\alpha_i=y\circ\alpha_i=0$ for all $i\in I\setminus I_0$. Then
\begin{align*}
\Tr^B(xy)=\sum_{i\in I_0}\wch\alpha^i\circ x\circ y\circ\alpha_i=\sum_{i,j\in I_0}\wch\alpha^i\circ x\circ\alpha_j\circ\wch\alpha^j\circ y\circ\alpha_i
\end{align*}
and similarly
\begin{align*}
\Tr^B(yx)=\sum_{i,j\in I_0}\wch\alpha^j\circ y\circ\alpha_i\circ\wch\alpha^i\circ x\circ\alpha_j
\end{align*}
The two RHS's are equal in $B/[B,B]$, noting that $\wch\alpha^i\circ x\circ\alpha_j$ and $\wch\alpha^j\circ y\circ\alpha_i$ are both in $\End_{-,B}(B)\simeq B$.
\end{proof}

\begin{df}
Let $\phi:B\rightarrow\Cbb$ be a \textbf{symmetric linear functional (SLF)}, i.e., a linear map satisfying $\phi(ab)=\phi(ba)$ for all $a,b\in B$. The (left) \textbf{pseudotrace} associated to $\phi$ (and $M$), denoted by \pmb{$\Tr^\phi$}, is defined to be
\begin{align}
\Tr^\phi=\phi\circ\Tr^B:A\rightarrow\Cbb
\end{align}
It is an SLF on $A$.
\end{df}

Thus, for each $x\in A$ we have
\begin{align}\label{eq7}
\Tr^\phi(x)=\sum_{i\in I}\phi(\wch\alpha^i\circ x\circ\alpha_i(1_B))
\end{align}

\section{AUF algebras and projective covers of irreducibles}

\begin{df}\label{lb4}
An algebra $A$ is called \textbf{almost unital and finite-dimensional (AUF)} if there is a family of mutually orthogonal idempotents $(e_i)_{i\in\fk I}$ such that the following conditions hold:
\begin{itemize}
\item[(a)] For each $i,j\in\fk I$ we have $\dim e_iAe_j<+\infty$.
\item[(b)] $A=\sum_{i,j\in \fk I}e_iAe_j$. (That is, for each $x\in A$ one can find a finite subset $I\subset \fk I$ and a collection $(x_{i,j})_{i,j\in\fk I}$ such that $x=\sum_{i,j\in I}e_ix_{i,j}e_j$.) 
\end{itemize}
Note that (b) automatically implies $A=\bigoplus_{i,j\in \fk I}e_iAe_j$.
\end{df}

It is illuminating to view an element $x\in A$ as an $\fk I\times \fk I$ matrix whose $(i,j)$-entry is $e_ixe_j$.

\begin{rem}\label{lb48}
	Each AUF algebra $A$ is almost unital. 
\end{rem}
\begin{proof}
	For each $x_1,\cdots,x_n\in A$, we can find a subset $I_0\subset \fk I$ such that $x_1,\cdots,x_n\in e'Ae'$, where $e'=\sum_{i\in I_0}e_i$. By choosing $n=1$ and $x_1=x\in A$, we see $x=e'xe'$. By choosing idempotents $x_i=e_i\in A$, we see $e_i\leq e'$ for all $1\leq i\leq n$.
\end{proof}

\begin{lm}\label{lb12}
In Def. \ref{lb4}, one can assume moreover that each $e_i$ is primitive (in $A$).
\end{lm}

\begin{proof}
Let $(e_i)_{i\in\fk I}$ be as in Def. \ref{lb4}. For each $i\in\fk I$, since $e_iAe_i$ is a finite-dimensional left $e_iAe_i$-module, it is a finite direct sum of indecomposable left $e_iAe_i$-submodules. By Cor. \ref{lb3} and \ref{lb5}, we have a finite direct sum $e_iAe_i=\bigoplus_{k\in\fk K_i} e_iAf_{i,k}$ where $(f_{i,k})_{k\in\fk K_i}$ is a finite family of mutually orthogonal idempotents in $e_iAe_i$, that $\sum_k f_{i,k}=e_i$, and that each $f_{i,k}$ is primitive in $e_iAe_i$. Clearly $f_{i,k}$ is also primitive in $A$. Replacing $(e_i)_{i\in\fk I}$ by $(f_{i,k})_{i\in\fk I,k\in\fk K_i}$ does the job.
\end{proof}

In the remaining part of this section, we always assume that $A$ is AUF.

\begin{rem}\label{lb14}
For any idempotents $e,f\in A$, we have
\begin{align*}
\dim eAf<+\infty
\end{align*}
Indeed, one can find a finite set $I_0\subset \fk I$ such that $e,f\in e'Ae'$ where $e'=\sum_{i\in I_0}e_i$. Then $\dim e'Ae'<+\infty$, and hence $\dim eAf<+\infty$.

It follows that each idempotent $e\in A$ has a (finite) orthogonal primitive decomposition $e=\eps_1+\cdots+\eps_n$. This follows from a decomposition of the finite-dimensional left $eAe$-module $eAe$ into indecomposable submodules.   \hfill\qedsymbol
\end{rem}

Recall Rem. \ref{lb10} about irreducibility.

\begin{thm}\label{lb15}
The following are true.
\begin{enumerate}
\item For each primitive idempotent $e\in A$, let $\rad(Ae)$ be the unique proper maximal left submodule of $Ae$ (cf. Cor. \ref{lb11}). Then $Ae\rightarrow Ae/\rad(Ae)$ gives a projective cover of the irreducible coherent module $Ae/\rad(Ae)$.
\item Any irreducible $M\in\QC(A)$ is isomorphic to $Ae/\rad(Ae)$ for some primitive idempotent $e\in A$.
\item Let $e,f$ be primitive idempotents. Then the following are equivalent:
\begin{enumerate}[label=(\arabic*)]
\item $Ae\simeq Af$ as left $A$-modules.
\item $Ae/\rad(Ae)\simeq Af/\rad(Af)$ as left $A$-modules.
\item $e\simeq f$, i.e., there is a partial isometry (in $A$) from $e$ to $f$. 
\end{enumerate}
\end{enumerate} 
\end{thm}

\begin{proof}
Part 1 was already proved, cf. Thm. \ref{lb8}. (Note that Thm. \ref{lb8} and its consequences are applicable since $\dim eAe<+\infty$ by Rem. \ref{lb14}.)

Part 2: By Prop. \ref{lb19}, $M$ has an epimorphism $\Psi$ from $A$. Let $(e_i)_{i\in\fk I}$ be as in Def. \ref{lb4} such that each $e_i$ is primitive (Lem. \ref{lb12}). Then $A\simeq\bigoplus_i Ae_i$ as left $A$-modules. The restriction of $\Psi$ to some $Ae_i$ must be nonzero, and hence must be surjective. Therefore $M\simeq Ae_i/\rad(Ae_i)$.

Part 3: (1)$\Rightarrow$(2) is obvious. (2)$\Rightarrow$(1) follows from the uniqueness of projective covers (Cor. \ref{lb22}). (1)$\Leftrightarrow$(3) follows from Cor. \ref{lb13}.
\end{proof}

\begin{co}\label{lb25}
Let $P\in\Coh(A)$. The following are equivalent.
\begin{enumerate}[label=(\arabic*)]
\item $P$ is projective and indecomposable.
\item $P$ is the projective cover of an irreducible $M\in\QC(A)$, and hence (by Thm. \ref{lb15}) $P$ is isomorphic to $Ae$ for some primitive idempotent $e\in A$.
\end{enumerate}
\end{co}

\begin{proof}
(2)$\Rightarrow$(1): This follows from Prop. \ref{lb24}.

(1)$\Rightarrow$(2): By Lem. \ref{lb7}, $P$ has an epimorphism to an irreducible, which (by Thm. \ref{lb15}) is of the form $Ae/\rad(Ae)$ where $e\in A$ is a primitive idempotent. We know that $Ae$ is its projective cover. Since $P$ is projective, by Prop. \ref{lb23}, $Ae$ is a direct summand of $P$. Since $P$ is indecomposable, we must have $P\simeq Ae$.
\end{proof}

\section{Pseudotraces and generating idempotents of strongly AUF algebras}

Let $A$ be AUF. In this section, we show that if $e\in A$ is a generating idempotent, any SLF $\psi$ on $A$ can be recovered from $\psi|_{eAe}$ via the pseudotrace construction. 

\begin{df}\label{lb28}
An idempotent $e\in A$ is called \textbf{generating} if every irreducible $M\in\QC(A)$ has an epimorphism from $Ae$.
\end{df}

\begin{pp}\label{lb16}
Let $e\in A$ be an idempotent. Let $e=\eps_1+\cdots+\eps_n$ be an orthogonal primitive decomposition (cf. Rem. \ref{lb14}). The following are equivalent:
\begin{enumerate}[label=(\arabic*)]
\item $e$ is generating.
\item Any primitive idempotent of $A$ is equivalent to $\eps_i$ for some $i$.
\item Any irreducible $M\in\QC(A)$ is isomorphic to $A\eps_i/\rad(A\eps_i)$ for some $i$.
\end{enumerate}
\end{pp}

\begin{proof}
(1)$\Rightarrow$(3): Each irreducible $M\in\QC(A)$ has an epimorphism from $Ae=A\eps_1\oplus\cdots\oplus A\eps_n$, and hence an epimorphism from some $A\eps_i$. By Cor. \ref{lb11}, the kernel of this epimorphism is $\rad(A\eps_i)$. Therefore, we have $A\eps_i/\rad(A\eps_i)\simeq M$. 

(3)$\Rightarrow$(1): Obvious.

(2)$\Leftrightarrow$(3): Immediate from Thm. \ref{lb15}.
\end{proof}

\begin{co}\label{lb59}
Let $e,f\in A$ be idempotents such that $e$ is a generating idempotent of $A$ and $e\leq f$. Then $e$ is a generating idempotent of $fAf$.
\end{co}

\begin{proof}
Let $p$ be any primitive idempotent of $fAf$. Then $p$ is a primitive idempotent of $A$. By Prop. \ref{lb16}, if we let $e=\eps_1+\cdots+\eps_n$ be an orthogonal primitive decomposition, then there exist $1\leq i\leq n$ and $u\in \eps_i A p,v\in pA\eps_i$ such that $uv=\eps_i$ and $vu=p$. So $p$ is equivalent in $fAf$ to $\eps_i$. By Prop. \ref{lb16}, we conclude that $e$ is generating in $fAf$.
\end{proof}

\begin{co}\label{lb33}
The following are equivalent.
\begin{enumerate}[label=(\arabic*)]
\item $A$ has a generating idempotent.
\item $\QC(A)$ has finitely many isomorphism classes of irreducible objects.
\item $A$ has finitely many equivalence classes of primitive idempotents.
\end{enumerate}
If one of these conditions holds, we say that $A$ is \textbf{strongly AUF}.
\end{co}

\begin{proof}
(1)$\Rightarrow$(2): Immediate from Prop. \ref{lb16}.

(2)$\Leftrightarrow$(3): Immediate from Thm. \ref{lb15}.

(2)$\Rightarrow$(1): Let $M_1,\dots,M_n\in\QC(A)$ exhaust all isomorphism classes of irreducibles. Let $(e_i)_{i\in\fk I}$ be as in Def. \ref{lb4}. For each $1\leq k\leq n$, by Prop. \ref{lb19}, $M_k$ has an epimorphism from $A$. Since $A=\bigoplus_{i\in\fk I}Ae_i$, it follows that $M_k$ has an epimorphism from $Ae_{i_k}$ for some $i_k\in\fk I$. If we assume at the beginning that $M_1,\dots,M_n$ are mutually non-isomorphic, then $e_{i_1},\dots,e_{i_n}$ must be distinct, and hence mutually orthogonal. So $e=e_{i_1}+\cdots+e_{i_n}$ is a generating idempotent.
\end{proof}

\begin{thm}\label{lb17}
Assume that $A$ is strongly AUF, and let $e\in A$ be a generating idempotent. Then the $A$-$(eAe)$ bimodule $Ae$ has a left coordinate system. In particular, by Rem. \ref{lb20}, $Ae$ is a projective right $eAe$-module.
\end{thm}

The following construction of a left coordinate system is important and is motivated by \cite[Lem. 3.9]{Ari10}.

\begin{proof}
Let $(e_i)_{i\in\fk I}$ be as in Def. \ref{lb4}. By Lem. \ref{lb12}, we can assume that each $e_i$ is primitive. Let $e=\eps_1+\cdots+\eps_n$ be an orthogonal primitive decomposition of $e$. By Prop. \ref{lb16}, there are partial isometries $u_i,v_i$ such that
\begin{gather*}
v_iu_i=\eps_{k_i}\qquad u_iv_i=e_i\\
u_i\in e_iA\eps_{k_i}\qquad v_i\in \eps_{k_i}Ae_i
\end{gather*} 
where $k_i\in\{1,\dots,n\}$. In particular $u_i\in e_iAe$ and $v_i\in eAe_i$. Let
\begin{gather*}
\alpha_i\in\Hom_{-,eAe}(eAe,Ae)\qquad \wch\alpha^i\in\Hom_{-,eAe}(Ae,eAe) \\
\alpha_i(exe)=u_i\cdot exe\qquad \wch\alpha^i(xe)=v_i\cdot xe 
\end{gather*}
One checks easily that $(\alpha_i,\wch\alpha^i)_{i\in\fk I}$ is a left coordinate system.
\end{proof}

The proof of \cite[Thm. 3.10]{Ari10} can be easily adapted to prove the following theorem.

\begin{thm}\label{lb42}
Assume that $A$ is strongly AUF, and let $e\in A$ be a generating idempotent. Then there is a linear isomorphism
\begin{align*}
\SLF(A)\xlongrightarrow{\simeq} \SLF(eAe)\qquad \psi\mapsto \psi|_{eAe}
\end{align*}
whose inverse is given by
\begin{align*}
\SLF(eAe)\xlongrightarrow{\simeq}\SLF(A)\qquad \phi\mapsto \Tr^\phi
\end{align*}
Here, $\Tr^\phi$ is the pseudotrace on $A$ with respect to $\phi$ and the $A$-($eAe$) bimodule $Ae$.
\end{thm}

\begin{proof}
Let $u_i,v_i,\alpha_i,\wch\alpha^i$ be as in the proof of Thm. \ref{lb17}. For any $\phi\in\SLF(eAe)$, let us compute $\Tr^\phi$. Let $x\in A$, viewed as an element of $\End_{-,eAe}(Ae)$. Then $\wch\alpha^i\circ x\circ\alpha_i\in\End_{-,eAe}(eAe)$ equals (the left multiplication by) $v_ixu_i$. Then
\begin{align}
\Tr^\phi(x)=\sum_{i\in\fk I}\phi(v_ixu_i)
\end{align}
Note that the RHS is a finite sum since $u_i=e_iu_i$, and since $xe_i=0$ for all but finitely many $i$. 

To show that $\Tr^\phi|_{eAe}=\phi$, we compute
\begin{align*}
\Tr^\phi(exe)=\sum_i\phi(v_iexeu_i)=\sum_i\phi(v_iexe\cdot eu_i)
\end{align*}
Since $v_iexe,eu_i\in eAe$, and since $\phi$ is SLF, we have
\begin{align*}
\Tr^\phi(exe)=\sum_i\phi(eu_i\cdot v_iexe)=\sum_i\phi(ee_iexe)=\phi(exe)
\end{align*}

Finally, let $\psi\in\SLF(A)$. Then for each $x\in A$,
\begin{align*}
\Tr^{\psi|_{eAe}}(x)=\sum_i\psi|_{eAe}(v_ixu_i)=\sum_i\psi(v_ixu_i)=\sum_i\psi(u_iv_ix)=\sum_i\psi(e_ix)=\psi(x)
\end{align*}
This proves $\Tr^{\psi|_{eAe}}=\psi$.
\end{proof}

\section{Projective generators of strongly AUF algebras}

Let $A$ be an AUF algebra.

\begin{rem}\label{lb35}
A left $A$-module $M$ is coherent if and only if $M$ is a quotient module of $(Ae)^{\oplus n}$ where $n\in\Zbb_+$ and $e\in A$ is an idempotent.
\end{rem}

\begin{proof}
``$\Leftarrow$" is obvious. Conversely, let $M\in\Coh(A)$. By Def. \ref{lb26}, $M$ is a quotient module of $Ap_1\oplus\cdots\oplus Ap_n$ where each $p_i$ is an idempotent. By Rem. \ref{lb48}, one can find an idempotent $e\in A$ which is $\geq p_1,\dots,p_n$. Then $M$ is a quotient module of $(Ae)^{\oplus n}$.
\end{proof}

\begin{rem}\label{lb56}
By Rem. \ref{lb35}, if $M\in\Coh(A)$ and $x\in A$, then $\dim xM<+\infty$.
\end{rem}

\begin{proof}
Suppose that $M$ has an epimorphism from $N:=(Ae)^{\oplus n}$ where $e\in A$ is an idempotent. Then $\dim xM\leq \dim xN$. Let $f\in A$ be an idempotent such that $x=fxf$. Then $xAe\subset fAe$, and hence
\begin{align*}
\dim xN=n\dim xAe\leq n\dim fAe<+\infty
\end{align*} 
\end{proof}

\subsection{Basic facts}

\begin{df}
Let $\scr S$ and $\scr T$ be classes of objects in $\Coh(A)$. We say that $\scr S$ \textbf{generates} $\scr T$ if each object of $\scr T$ is a quotient of a \textit{finite} direct sum of objects in $\scr S$.
\end{df}

\begin{df}
We say that $M\in\Coh(A)$ is a \textbf{generator} (of $\Coh(A)$) if it generates every object of $\Coh(A)$, i.e., every $N\in\Coh(A)$ is a quotient module of $M^{\oplus n}$ for some $n\in\Zbb_+$. A generator which is also projective is called a \textbf{projective generator}.
\end{df}

\begin{eg}\label{lb34}
Let $(e_i)_{i\in\fk I}$ be as in Def. \ref{lb4}. Then $\scr S:=\{Ae_i:i\in\fk I\}$ generates $\Coh(A)$.
\end{eg}

\begin{proof}
By the proof of Rem. \ref{lb48}, for any idempotent $e\in A$ one can find a finite set $I_0\subset \mathfrak I$ such that $e\leq\sum_{i\in I_0} e_i$. Therefore, $\scr S$ generates each $Ae$, and hence (by Rem. \ref{lb35}) generates $\Coh(A)$.
\end{proof}

\begin{pp}\label{lb27}
Let $M\in\Coh(A)$ be projective. The following are equivalent.
\begin{enumerate}[label=(\arabic*)]
\item $M$ is a projective generator.
\item Each irreducible $N\in\Coh(A)$ has an epimorphism from $M$.
\end{enumerate}
\end{pp}

\begin{proof}
(1)$\Rightarrow$(2): Obvious.

(2)$\Rightarrow$(1): Let $(e_i)_{i\in\fk I}$ be as in Def. \ref{lb4}. By Lem. \ref{lb12}, we assume that each $e_i$ is primitive. By Exp. \ref{lb34}, it suffices to prove that $M$ generates each $Ae_i$. By Thm. \ref{lb15}, $Ae_i$ is the projective cover of the irreducible $N:=Ae_i/\rad(Ae_i)$. By (2), $M$ has an epimorphism to $N$. Since $M$ is projective, by Prop. \ref{lb23}, $Ae_i$ is isomorphic to a direct summand of $M$.
\end{proof}

\begin{comment}
\begin{pp}
Let $\mc E$ be a set of idempotents of $A$. Then the following are equivalent.
\begin{enumerate}[label=(\arabic*)]
\item $\scr S:=\{Ae:e\in\mc E\}$ generates $\Coh(A)$.
\item For any primitive idempotent $p\in A$, there exists $e\in\mc E$ such that $p$ is equivalent to a sub-idempotent of $e$.\footnote{Namely, there is an idempotent $e'\in A$ such that $e'\leq e$ and $p\simeq e'$.} 
\end{enumerate}
\end{pp}

\begin{proof}
(2)$\Rightarrow$(1): By Def. \ref{lb26}, any $M\in\Coh(A)$ is a quotient of a finite direct sum of modules of the form $Ap$ where $p\in A$ is an idempotent. Therefore, it suffices to prove that $\scr S$ generates each $Ap$. Since $\dim pAp<+\infty$ (Rem. \ref{lb14}), $p$ is a finite sum of primitive idempotents. Therefore, it suffices to assume that $p$ is primitive. Then there exists $e\in\mc E$ such that $p$ is isomorphic to a sub-idempotent of $e$. Therefore $Ap$ is a quotient of $Ae$. This proves that $\scr S$ generates $Ap$.

(1)$\Rightarrow$(2): By Thm. \ref{lb15}, $N:=Ap/\rad(Ap)$ is irreducible with projective cover $Ap$. Since $\scr S$ generates $N$, $N$ has an epimorphism from a direct sum of members of $\scr S$. Since $N$ is irreducible, there exists $e\in\mc E$ and an epimorphism $Ae\rightarrow N$. Since $Ae$ is projective, by Prop. \ref{lb23}, $Ap$ is isomorphic to a direct summand of $Ae$. This is equivalent to saying that $p$ is isomorphic to a sub-idempotent of $e$.
\end{proof}

\end{comment}

\begin{co}\label{lb32}
Let $e\in A$ be an idempotent. Then the following are equivalent.
\begin{enumerate}[label=(\arabic*)]
\item $Ae$ is a (necessarily projective) generator.
\item $e$ is a generating idempotent.
\end{enumerate}
\end{co}

\begin{proof}
(1)$\Rightarrow$(2): Clear from Def. \ref{lb28}. (2)$\Rightarrow$(1): Immediate from Prop. \ref{lb27}.
\end{proof}

\begin{pp}\label{lb50}
$\Coh(A)$ has a projective generator if and only if $A$ is strongly AUF.
\end{pp}

\begin{proof}
``$\Leftarrow$" follows from Cor. \ref{lb33} and \ref{lb32}. Conversely, if $\Coh(A)$ has a projective generator $M$, by Rem. \ref{lb35}, an idempotent $e\in A$ can be found such that $Ae$ generates $M$, and hence generates $\Coh(A)$. So $e$ is a generating idempotent. Thus, by Cor. \ref{lb33}, $\Coh(A)$ has finitely many irreducibles. So $A$ is strongly AUF.
\end{proof}

\subsection{Projective generators and endomorphism algebras}

Our next goal is to give criteria for projective generators in terms of the endomorphism algebras. We need the endomorphism algebras to be finite-dimensional:

\begin{pp}\label{lb29}
Let $M,N\in\Coh(A)$. Then
\begin{align*}
\dim \Hom_{A,-}(M,N)<+\infty
\end{align*}
\end{pp}

\begin{proof}
By Def. \ref{lb26}, there is an epimorphism from a finite direct sum $\bigoplus_i Ae_i$ to $M$, where $e_i$ is an idempotent. By taking composition with this epimorphism, we get
\begin{align}
\Hom_{A,-}(M,N)\rightarrow \Hom_{A,-}\Big(\bigoplus_i Ae_i,N\Big)\simeq\bigoplus_i \Hom_{A,-}(Ae_i,N)
\end{align}
where the first map is injective. Thus, it suffices to prove that each $\Hom_{A,-}(Ae_i,N)$ is finite-dimensional.

Again, we can find an epimorphism  $\Phi:\bigoplus_jAf_j\twoheadrightarrow N$ (where $\bigoplus_j$ is finite). Since $Ae_i$ is projective, each $\alpha\in\Hom_{A,-}(Ae_i,N)$ can be lifted to some $\beta\in\Hom_{A,-}(Ae_i,\bigoplus_j Af_j)$ such that $\alpha=\Phi\circ\beta$. Thus
\begin{align*}
\dim\Hom_{A,-}(Ae_i,N)\leq \dim\Hom_{A,-}\Big(Ae_i,\bigoplus_j Af_j\Big)=\sum_j \dim\Hom_{A,-}(Ae_i,Af_j)
\end{align*}
where $\dim\Hom_{A,-}(Ae_i,Af_j)=\dim e_iAf_j<+\infty$.
\end{proof}

\begin{pp}\label{lb30}
Let $M$ be a left $A$-module. Let $B=\End_{A,-}(M)^\opp$, and let $p,q\in B$ be idempotents. Then an element of $\Hom_{A,-}(Mp,Mq)$ is precisely the right multiplication of an element of $pBq$. In particular, we have a canonical isomorphism
\begin{align*}
\End_{A,-}(Mp)^\opp\simeq pBp
\end{align*}
Consequently, the direct summands of the left $A$-module $Mp$ correspond bijectively to the sub-idempotents of $p$ in $B$.
\end{pp}

\begin{proof}
This is similar to the proofs of Prop. \ref{lb6} and Cor. \ref{lb3}. Any $y\in pBq$ defines a morphism $Mp\rightarrow Mq$ by right multiplication. Conversely, if $T\in\Hom_{A,-}(Mp,Mq)$, let $\wht T:M\rightarrow M$ be $\wht T(\xi)=T(\xi p)$. Then $\wht T\in\End_{A,-}(M)$, and hence $\wht T$ is the right multiplication by some $\wht y\in B$. Note that $T=\wht T|_{Mp}$, and hence $T(\xi p)=\xi p \wht y$ for each $\xi\in M$. Since $T$ has range in $Mq$, we have $T(\xi p)=\xi p\wht yq$. So $T$ is the right multiplication by $y:=p\wht yq\in pBq$. 
\end{proof}

\begin{thm}\label{lb31}
Let $M\in\Coh(A)$. Let $B=\End_{A,-}(M)^\opp$ which is a finite-dimensional unital algebra (by Prop. \ref{lb29}). Let $p\in B$ be an idempotent. Consider the following statements:
\begin{enumerate}[label=(\arabic*)]
\item As coherent left $A$-modules, $Mp$ generates $M$.
\item $p$ is a generating idempotent of $B$.
\end{enumerate}
Then (2)$\Rightarrow$(1). If $M$ is projective, then (1)$\Leftrightarrow$(2). 
\end{thm}

\begin{proof}
(2)$\Rightarrow$(1): Since $\dim B<+\infty$, we have a primitive orthogonal decomposition $1_B=q_1+\cdots+q_n$ where each $q_j\in B$ is a primitive idempotent. By Prop. \ref{lb16}, each $q_j$ is equivalent to a sub-idempotent of $p$. Thus $Mq_j$ is isomorphic to a direct summand of the left $A$-module $Mp$. So $Mp$ generates $\bigoplus_j Mq_j=M$.

(1)$\Rightarrow$(2): Let $q$ be any primitive idempotent of $B$. Since $Mp$ generates $M$ and since $M$ generates $Mq$, we have that $Mp$ generates $Mq$. We claim that $Mq$ is isomorphic to a direct summand of $Mp$. Then Prop. \ref{lb30} will imply that $q$ is equivalent (in $B$) to a sub-idempotent of $p$. This implies (2), thanks to Prop. \ref{lb16}.

Let us prove the claim, assuming that $M$ is projective. Since $Mq$ is a direct summand of $M$, we see that $Mq$ is projective. Since $q$ is primitive in $B$, by Prop. \ref{lb30}, $Mq$ is an indecomposable left $A$-module. Therefore, by Cor. \ref{lb25}, $Mq$ is the projective cover of an irreducible $N\in\Coh(A)$. Since $Mp$ generates $Mq$, it generates $N$. Thus $N$ has an epimorphism from a finite direct sum of $Mp$. Since $N$ is irreducible, $N$ has an epimorphism from $Mp$. Note that $Mp$ is also projective. Therefore, by Prop. \ref{lb23}, $Mq$ is isomorphic to a direct summand of $Mp$.
\end{proof}

%Using Thm. \ref{lb31}, one obtains the following criterion for projective generators.

\begin{co}\label{lb43}
Assume that $G\in\Coh(A)$ is a projective generator. Let $M$ be a left $A$-module. Then the following are equivalent.
\begin{enumerate}[label=(\arabic*)]
\item $M\in\Coh(A)$, and $M$ is a projective generator (of $\Coh(A)$).
\item There exist $n\in\Zbb_+$ and a generating idempotent $p$ of $B:=\End_{A,-}(G^{\oplus n})^\opp$ such that $M\simeq G^{\oplus n}\cdot p$.
\end{enumerate}
\end{co}

In particular, if $e\in A$ is a generating idempotent, one can take $G=Ae$. Thus a projective generator of $\Coh(A)$ is (up to isomorphisms) precisely of the form $(Ae)^{\oplus n}p$ where $n\in\Zbb_+$ and $p\in \End_{A,-}((Ae)^{\oplus n})^\opp$
is a generating idempotent.

\begin{proof}
(2)$\Rightarrow$(1): By Thm. \ref{lb31}, $M$ generates $G^{\oplus n}$. So $M$ is a generator. Since $G^{\oplus n}p$ is a direct summand of the projective coherent module $G^{\oplus n}$, $G^{\oplus n}p$ is also projective and coherent.

(1)$\Rightarrow$(2): $M$ has an epimorphism from $G^{\oplus n}$ for some $n\in\Zbb_+$. Since $M$ is projective, this epimorphism splits. So $M$ can be viewed as a direct summand of $G^{\oplus n}$. Let $p$ be the projection of $G^{\oplus n}$ onto $M$, which can be viewed as an endomorphism of $G^{\oplus n}$. So $p$ is an idempotent of $B$, and $M=G^{\oplus n}p$. Since $M$ is a generator, it generates $G^{\oplus n}$. Since $G^{\oplus n}$ is projective, by Thm. \ref{lb31}, $p$ is generating.
\end{proof}

\section{Right pseudotraces}

Let $A$ be an AUF algebra. Let $B$ be a unital algebra. Let $M$ be an $A$-$B$ bimodule, coherent as a left $A$-module. 

For each $y\in B$ and $\xi\in M$, we write $\xi y$ as $y^\opp \xi$. Namely, $y^\opp$ is viewed as an element of $\End_{A,-}(M)$.

\begin{df}\label{lb36}
A \textbf{right coordinate system} of $M$ denotes a collection of morphisms
\begin{align*}
\beta_j\in\Hom_{A,-}(Ae,M)\qquad \wch\beta^j\in \Hom_{A,-}(M,Ae)
\end{align*}
where $e\in A$ is an idempotent (called the \textbf{domain idempotent}), and $j$ runs through a \textit{finite} index set $J$ such that the $\sum_{j\in J}\beta_j\circ\wch\beta^j$ equals $\id_M$.
\end{df}

\begin{rem}\label{lb38}
	$M$ has a right coordinate system iff $M$ is $A$-projective.
\end{rem}

\begin{proof}
By Rem. \ref{lb35}, each  $N\in\Coh(A)$ has an epimorphism from $(Ae)^{\oplus n}$ where $e\in A$ is an idempotent and $n\in\Zbb_+$. This epimorphism splits iff $N$ is projective in $\Coh(A)$. Therefore, similar to Rem. \ref{lb20}, we see that $M$ has a right coordinate system iff $M$ is $A$-projective.
\end{proof}

\begin{rem}
In Def. \ref{lb36}, one can freely enlarge the domain idempotent $e$. More precisely, suppose that $f\in A$ is an idempotent such that $e\leq f$. One can define a new right coordinate system
\begin{gather}\label{eq4}
\begin{gathered}
\gamma_j\in\Hom_{A,-}(Af,M)\qquad \wch\gamma^j\in \Hom_{A,-}(M,Af)\\
\gamma_j(af)=\beta_j(ae)\qquad \wch\gamma^j(\xi)=\wch\beta^j(\xi)
\end{gathered}
\end{gather}
called the  \textbf{canonical extension} of $(\beta_j,\wch\beta^j)_{j\in J}$.
\end{rem}

\begin{df}
Assume that $M$ has a right coordinate system $(\beta_j,\wch\beta^j)_{j\in J}$. For each $\psi\in\SLF(A)$, define the (right) \textbf{pseudotrace} \pmb{${}^\psi\Tr$} associated to $\psi$ to be
\begin{gather*}
{^\psi}\Tr:B\rightarrow\Cbb\qquad
{^\psi}\Tr(y)=\sum_{j\in J}\psi\big((\wch\beta^j\circ y^\opp\circ\beta_j)^\opp\big)
\end{gather*}
noting that $\wch\beta^j\circ y^\opp\circ\beta_j\in\End_{A,-}(Ae)\simeq (eAe)^\opp$. In other words,
\begin{align}\label{eq5}
{^\psi}\Tr(y)=\sum_{j\in J}\psi(\wch\beta^j\circ y^\opp\circ\beta_j(e))
\end{align}
\end{df}

Note that in \eqref{eq5} we have $\beta_j(e)\in M$, and hence $\wch\beta^j\circ y^\opp\circ\beta_j(e)\in Ae$. So
\begin{align*}
\wch\beta^j\circ y^\opp\circ\beta_j(e)=\wch\beta^j\circ y^\opp\circ\beta_j(e^2)=e\wch\beta^j\circ y^\opp\circ\beta_j(e)\in eAe
\end{align*}

\begin{pp}
	Assume that $M$ is $A$-projective. Let $\psi\in\SLF(A)$. Then ${^\psi}\Tr\in\SLF(B)$. Moreover, the definition of ${^\psi}\Tr$ is independent of the choice of right coordinate systems.
\end{pp}

\begin{proof}
From \eqref{eq4} and \eqref{eq5}, it is clear that a canonical extension of the right coordinate system does not affect the value of ${^\psi}\Tr(y)$. Also, note that since $A$ is AUF, for any idempotents $e_1,e_2\in A$ there is an idempotent $e_3$ such that $e_1,e_2\leq e_3$. Therefore, to compare ${^\psi}\Tr$ defined by two coordinate systems $(\alpha_\blt,\wch\alpha^\blt)$ and $(\beta_\star,\wch\beta^\star)$, by performing canonical extensions, it suffices to assume that their domain idempotents are equal. Then one can use the same argument as in Lem. \ref{lb9} to show that $(\alpha_\blt,\wch\alpha^\blt)$ and $(\beta_\star,\wch\beta^\star)$ define the same ${^\psi}\Tr$. Finally, similar to the proof of Prop. \ref{lb37}, one shows that ${^\psi}\Tr$ is symmetric.
\end{proof}

\begin{eg}
Let $M=Ae$ and $B=eAe$ where $e\in A$ is an idempotent. Then the identity map on $Ae$ gives a right coordinate system. From this, one sees that if $\psi\in\SLF(A)$ then
\begin{align*}
{^\psi}\Tr=\psi|_{eAe}
\end{align*}
\end{eg}

\begin{eg}\label{lb47}
More generally, let $M=(Ae)^{\oplus n}$ and $B=\End_{A,-}(M)^\opp$. So $B=eAe\otimes\Cbb^{n\times n}$. Let
\begin{align*}
\trc:\Cbb^{n\times n}\rightarrow\Cbb
\end{align*}
be the standard trace on $\Cbb^{n\times n}$. A right coordinate system can be chosen to be the $n$ canonical embeddings $Ae\rightarrow (Ae)^{\oplus n}$ and the $n$ canonical projections $(Ae)^{\oplus n}\rightarrow Ae$. Then one easily sees that
\begin{align*}
{^\psi}\Tr=\psi|_{eAe}\otimes\trc
\end{align*}
\end{eg}

\begin{pp}\label{lb40}
	Assume that $M$ is $A$-projective. Let $p\in B$ be an idempotent. Let $\psi\in\SLF(A)$. Let ${}^\psi\Tr_M:B\rightarrow\Cbb$ be the pseudotrace associated to $M$. Then the pseudotrace ${}^\psi\Tr_{Mp}:pBp\rightarrow\Cbb$ associated to the $A$-$(pBp)$ bimodule $Mp$ is equal to ${}^\psi\Tr_M\big|_{pBp}$, i.e.
\begin{align*}
{}^\psi\Tr_{Mp}={}^\psi\Tr_M\big|_{pBp}
\end{align*}
\end{pp}

\begin{proof}
Let $(\beta_\blt,\wch\beta^\blt)$ be a right coordinate system (with domain idempotent $e\in A$) as in Def. \ref{lb36}. Then one has a right coordinate system
\begin{gather*}
\gamma_j\in\Hom_{A,-}(Ae,Mp)\qquad \wch\gamma^j\in \Hom_{A,-}(Mp,Ae)\\
\gamma_j(ae)=\beta_j(ae)p\qquad \wch\gamma^j(\xi p)=\wch\beta^j(\xi p)
\end{gather*}
noting that $Mp\leq M$, and hence $\wch\gamma^j$ is simply the restriction of $\wch\beta^j$ to $Mp$. Using \eqref{eq5} one computes that for each $y\in B$,
\begin{align*}
&{}^\psi\Tr_{Mp}(pyp)=\sum_j \psi(\wch\gamma^j\circ (pyp)^\opp\circ \gamma_j(e))=\sum_j \psi(\wch\beta^j\circ (pyp)^\opp\circ \beta_j(e)p)\\
=&\sum_j \psi(\wch\beta^j\circ (pyp)^\opp\circ p^\opp\circ\beta_j(e))=\sum_j \psi(\wch\beta^j\circ (pyp)^\opp\circ\beta_j(e))={}^\psi\Tr_M(pyp)
\end{align*}
\end{proof}

\section{Equivalence of left and right pseudotraces}

Let $A,B$ be algebras where $B$ is unital and finite dimensional.

\subsection{Preliminary discussion}

In this subsection, assume that $A$ is AUF. We shall consider $M\in\Coh(A)$ such that the left and the right pseudotrace constructions are both available to the $A$-$(\End_{A,-}(M)^\opp)$ bimodule $M$. By Rem. \ref{lb38}, $M$ needs to be assumed  $A$-projective. One also needs $M$ to be $\End_{A,-}(M)^\opp$-projective. In fact, these two conditions are precisely what ensure that both left and right pseudotraces can be defined.

\begin{pp}
Let $M$ be an $A$-$B$ bimodule. Assume that $M$ is $A$-coherent. Then the following are equivalent.
\begin{enumerate}[label=(\arabic*)]
\item $M$ has a left coordinate system.
\item $M$ is $B$-projective.
\end{enumerate}
\end{pp}

Although this proposition will not be used in the current note, we include it here as it may be of use in the future.

\begin{proof}
(1)$\Rightarrow$(2): See Rem. \ref{lb20}.

(2)$\Rightarrow$(1): Let $(e_i)_{i\in \fk I}$ be as in Def. \ref{lb4}. By Rem. \ref{lb56}, each $e_iM$ is finite-dimensional. Therefore, the right $B$-module $e_iM$ has an epimorphism from $B^{\oplus n}$ which splits because $M$ is $B$-projective (and hence $e_iM$ is projective since $M=\bigoplus_{i\in\fk I}e_iM$). Therefore, for each $i\in\fk I$, there is a finite left coordinate system $\alpha_{i,\blt}\in\Hom_{-,B}(B,e_iM)$ and $\wch\alpha^{i,\blt}\in\Hom_{-,B}(e_iM,B)$. Let
\begin{gather*}
\gamma_{i,\blt}\in\Hom_{-,B}(B,M)\qquad \wch\gamma^{i,\blt}\in\Hom_{-,B}(M,B)\\
\gamma_{i,\blt}(b)=\alpha_{i,\blt}(b)\qquad\wch\gamma^{i,\blt}(\xi)=\wch\alpha^{i,\blt}(e_i\xi)
\end{gather*}
Then one checks easily that $(\gamma_{i,\blt},\wch\gamma^{i,\blt})_{i\in\fk I}$ is a left coordinate system of $M$.   
\end{proof}

\subsection{Calculation of some left pseudotraces}

In this subsection, $A$ is not assumed to be AUF. Let $M$ be an $A$-$B$ bimodule. 

The goal of this subsection is to prepare for the proof of the main Thm. \ref{lb44}.  The following theorem is dual to Prop. \ref{lb40}. 

\begin{thm}\label{lb45}
Assume that $M$ has a left coordinate system. Let $p\in B$ be a generating idempotent. Then the following are true.
\begin{enumerate}
\item The $A$-$(pBp)$ bimodule $Mp$ has a left coordinate system.
\item Let $\phi\in\SLF(B)$. Then on $A$, the pseudotrace associated to $\phi|_{pBp}$ and $Mp$ is equal to the pseudotrace associated to $\phi$ and $M$. Namely,
\begin{align}\label{eq6}
\Tr_{Mp}^{\phi|_{pBp}}=\Tr^\phi_M
\end{align}
\end{enumerate}
\end{thm}

In this theorem, we do not require that $A$ is AUF.

\begin{proof}
Choose a left coordinate system for $M$:
\begin{align*}
\alpha_i\in\Hom_{-,B}(B,M)\qquad \wch\alpha^i\in \Hom_{-,B}(M,B)\qquad i\in\fk I
\end{align*}
Since $p$ is generating, similar to the proof of Thm. \ref{lb17}, we can find finitely many elements $u_k,v_k$ in $B$ such that
\begin{gather*}
v_ku_k=p_k\qquad u_kv_k=q_k\\
u_k\in q_kBp_k\qquad v_k\in p_kBq_k
\end{gather*}
where $p_k,q_k\in B$ are idempotents, $1_B=\sum_k q_k$ is a primitive orthogonal decomposition of $1_B$, and $p_k\leq p$ for each $k$. \footnote{So $p_k,q_k$ are similar to $\eps_{k_i},e_i$ in the proof of Thm. \ref{lb17}.} Let
\begin{gather*}
\theta_{i,k}\in \Hom_{-,pBp}(pBp,Mp)\qquad \wch\theta^{i,k}\in\Hom_{-,pBp}(Mp,pBp)\\
\theta_{i,k}(pyp)=\alpha_i(u_k\cdot pyp)\qquad \wch\theta^{i,k}(\xi p)=v_k\cdot \wch\alpha^i(\xi p)
\end{gather*}
noting that $\alpha_i(u_k\cdot pyp)=\alpha_i(u_k)pyp\in Mp$ and $v_k\cdot \wch\alpha^i(\xi p)=v_k\cdot \wch\alpha^i(\xi)p\in p_kBp\subset pBp$. 

For each $\xi\in M$, note that if $\wch\alpha^i(\xi)=0$, then $\wch\theta^{i,k}(\xi p)=v_k\wch\alpha^i(\xi)p=0$. Therefore, $\wch\theta^{i,k}(\xi p)=0$ for all but finitely many $i$ and $k$. Moreover, we compute
\begin{align*}
&\sum_{i,k}\theta_{i,k}\circ\wch\theta^{i,k}(\xi p)=\sum_{i,k}\theta_{i,k}(v_k\wch\alpha^i(\xi p))=\sum_{i,k}\alpha_i(u_k v_k\wch\alpha^i(\xi p))\\
=&\sum_{i,k}\alpha_i(q_k\wch\alpha^i(\xi p))=\sum_{i}\alpha_i\circ\wch\alpha^i(\xi p)=\xi p
\end{align*}
where all the sums are finite. This proves that $(\theta,\wch\theta)$ satisfies Def. \ref{lb41}-(a). It is easy to check Def. \ref{lb41}-(b). So we have proved that $(\theta,\wch\theta)$ is a left coordinate system of $Mp$.

It remains to check \eqref{eq6}. Choose any $x\in A$. By \eqref{eq7} and the fact that $1_{pBp}=p$,
\begin{align*}
&\Tr_{Mp}^{\phi|_{pBp}}(x)=\sum_{i,k}\phi(\wch\theta^{i,k}\circ x\circ \theta_{i,k}(p))=\sum_{i,k}\phi(\wch\theta^{i,k}\circ x\circ \alpha_i(u_kp))\\
=&\sum_{i,k}\phi(\wch\theta^{i,k}\circ x\circ \alpha_i(u_k))=\sum_{i,k}\phi(v_k\cdot\wch\alpha^i(x\circ\alpha_i(u_k)))
\end{align*}
Since $\wch\alpha^i,x,\alpha_i$ commute with the right multiplication by $v_k$, and since $\phi$ is symmetric, 
\begin{align*}
&\Tr_{Mp}^{\phi|_{pBp}}(x)=\sum_{i,k}\phi(\wch\alpha^i(x\circ\alpha_i(u_k))v_k)=\sum_{i,k}\phi(\wch\alpha^i(x\circ\alpha_i(u_kv_k)))\\
=&\sum_{i}\phi(\wch\alpha^i(x\circ\alpha_i(1_B)))=\Tr_M^\phi(x)
\end{align*}
This finishes the proof of \eqref{eq6}.
\end{proof}

\begin{co}\label{lb46}
	Assume that $M$ has a left coordinate system. Let $n\in\Zbb_+$. Let $\wtd B=B\otimes\Cbb^{n\times n}$. Then the $A$-$\wtd B$ bimodule $M^{\oplus n}\simeq M\otimes\Cbb^{1,n}$ has a left coordinate system. Moreover, for each $\phi\in\SLF(B)$, we have
\begin{align}\label{eq8}
\Tr^{\phi\otimes\trc}_{M^{\oplus n}}=\Tr^\phi_M
\end{align}
as pseudotraces on $A$ associated to $\phi\otimes\trc\in\SLF(\wtd B)$ and $\phi\in\SLF(B)$, respectively.
\end{co}

Recall that $\trc\in\SLF(\Cbb^{n\times n})$ is the standard trace on the $n\times n$ matrix algebra.

\begin{proof}
Choose a left coordinate system
\begin{gather*}
\alpha_i\in\Hom_{-,B}(B,M)\qquad \wch\alpha^i\in \Hom_{-,B}(M,B)
\end{gather*}
of $M$. Define
\begin{gather*}
\gamma_i\in\Hom_{-,\wtd B}(\wtd B,M^{\oplus n})\qquad \wch\gamma^i\in\Hom_{-,\wtd B}(M^{\oplus n},\wtd B)
\end{gather*}
such that
\begin{gather*}
\gamma_i
\begin{bmatrix}
y_{1,1}&\cdots&y_{1,n}\\
&\vdots&\\
y_{n,1}&\cdots&y_{n,n}
\end{bmatrix}
=[\alpha_i(1_B),0,\dots,0]\begin{bmatrix}
y_{1,1}&\cdots&y_{1,n}\\
&\vdots&\\
y_{n,1}&\cdots&y_{n,n}
\end{bmatrix}=[\alpha_i(y_{1,1}),\dots,\alpha_i(y_{1,n})]
\\
\wch\gamma^i[\xi_1,\dots,\xi_n]
=
\begin{bmatrix}
\wch\alpha^i(\xi_1)&\cdots&\wch\alpha^i(\xi_n)\\
0&\cdots&0\\
\vdots&\vdots&\vdots\\
0&\cdots&0
\end{bmatrix}
\end{gather*}
One checks easily that this is a left coordinate system of $M^{\oplus n}$. Now \eqref{eq8} follows by applying Thm. \ref{lb45} to the $A$-$\wtd B$ bimodule $M^{\oplus n}$ and the generating projection $p\in\wtd B$, where $p$ is the matrix whose $(1,1)$-entry is $1$ and other entries are $0$. 
\end{proof}

\subsection{The main theorem}

Assume that $A$ is strongly AUF (cf. Cor. \ref{lb33}) so that $A$ has a projective generator (cf. Prop. \ref{lb50}). The following generalization of Thm. \ref{lb42} is the main theorem of this note.

\begin{thm}\label{lb44}
Assume that $M\in\Coh(A)$ is a projective generator. Assume that $B=\End_{A,-}(M)^\opp$ so that $M$ is an $A$-$B$ bimodule. Then $M$ has left and right coordinate systems. Moreover, we have a linear isomorphism
\begin{subequations}\label{eq9}
\begin{align}\label{eq9a}
\SLF(A)\xrightarrow{\simeq}\SLF(B)\qquad \psi\mapsto {^\psi}\Tr
\end{align}
whose inverse map is
\begin{align}\label{eq9b}
\SLF(B)\xrightarrow{\simeq}\SLF(A)\qquad \phi\mapsto \Tr^\phi
\end{align}
\end{subequations}
\end{thm}

Of course, both pseudotraces are associated to $M$; we have suppressed the subscript $M$.

\begin{proof}
Note that $\dim B<+\infty$ by Prop. \ref{lb29}. So $\dim\SLF(B)<+\infty$. Since $M\in\Coh(A)$ is $A$-projective, by Rem. \ref{lb38}, $M$ has a right coordinate system.  By Cor. \ref{lb43}, we may assume that $M=G\cdot p$ where 
\begin{itemize}
\item $G=(Ae)^{\oplus n}$ for some $n\in\Zbb_+$ and a generating idempotent $e\in A$.
\item $M=Gp$ where $p$ is a generating idempotent of $\wtd B=\End_{A,-}(G)^\opp=eAe\otimes\Cbb^{n\times n}$.
\item $B=p\wtd Bp$ (by Prop. \ref{lb30}).
\end{itemize}
By Thm. \ref{lb17} and Cor. \ref{lb46}, $G$ has a left coordinate system. Therefore, by Thm. \ref{lb45}, $M$ has a left coordinate system.

By Thm. \ref{lb42}, we have $\dim\SLF(A)=\dim\SLF(eAe)$. Clearly we have a linear isomorphism
\begin{align*}
\SLF(eAe)\xrightarrow{\simeq}\SLF(eAe\otimes\Cbb^{n\times n})\qquad \omega\mapsto\omega\otimes\trc
\end{align*}
So $\dim\SLF(eAe)=\dim\SLF(\wtd B)$. By Thm. \ref{lb42}, we have $\dim\SLF(\wtd B)=\dim\SLF(B)$. This proves $\dim\SLF(A)=\dim\SLF(B)<+\infty$.

Choose any $\psi\in\SLF(A)$. By Exp. \ref{lb47}, ${}^\psi\Tr_G:\wtd B\rightarrow\Cbb$ equals $\psi|_{eAe}\otimes\trc$. By Prop. \ref{lb40}, on $B=p(eAe\otimes\Cbb^{n\times n})p$ we have
\begin{align*}
{}^\psi\Tr_M=(\psi|_{eAe}\otimes\trc)\big|_B\qquad=:\phi
\end{align*}
Now $\phi\in\SLF(B)$. By Thm. \ref{lb45} and Cor. \ref{lb46},
\begin{align*}
\Tr_M^\phi=\Tr_{Gp}^{(\psi|_{eAe}\otimes\trc)|_B}=\Tr_G^{\psi|_{eAe}\otimes\trc}=\Tr_{Ae}^{\psi|_{eAe}}
\end{align*}
By Thm. \ref{lb42}, $\Tr_{Ae}^{\psi|_{eAe}}=\psi$. So $\Tr_M^\phi=\psi$. We have thus proved that $\eqref{eq9b}\circ\eqref{eq9a}$ is the identity map on $\SLF(A)$. This finishes the proof.
\end{proof}

\section{Equivalence of non-degeneracy of left and right pseudotraces}

\begin{df}
Let $A$ be an algebra and $\psi\in\SLF(A)$. We say that $\psi$ is \textbf{non-degenerate} if
\begin{align*}
\{x\in A:\psi(xA)=0\}\equiv\{x\in A:\psi(xy)=0,\forall y\in A\}
\end{align*}
is zero.
\end{df}

In the following, $A$ is always assumed to be AUF.

\begin{lm}\label{lb60}
Let $e\in A$ be an idempotent, and let $\psi\in\SLF(A)$. If $\psi$ is non-degenerate, then the restriction $\psi|_{eAe}$ is non-degenerate. Conversely, if $\psi|_{eAe}$ is non-degenerate and $e$ is generating, then $\psi$ is non-degenerate.
\end{lm}

\begin{proof}
Assume that $\psi$ is non-degenerate. Choose $x\in eAe$ such that $\psi(xeAe)=0$. Then
\begin{align*}
\psi(xA)=\psi(exeA)=\psi(xeAe)=0
\end{align*}
and hence $x=0$. Therefore $\psi|_{eAe}$ is non-degenerate.

Conversely, assume that $\psi|_{eAe}$ is non-degenerate and $e$ is generating. Choose $x\in A$ such that $\psi(xA)=0$. Then for each $a,b\in A$,
\begin{align*}
\psi(eaxbe\cdot eAe)=\psi(eaxbeAe)=\psi(xbeAea)=0
\end{align*}
Therefore $eaxbe=0$. Since $b$ is arbitrary, we have  $eaxAe=0$. Since $e$ is generating, it is not hard to show that the left $A$-module $Ae$ is faithful. (See for example Lem. \ref{lb54}.) It follows that $eax=0$. Therefore $eAx=0$. Similarly, $eA$ is a faithful right $A$-module. Hence $x=0$. This proves the non-degeneracy of $\psi$.
\end{proof}

\begin{pp}
Assume that $\psi\in\SLF(A)$ is non-degenerate. Let $M\in\Coh(A)$ be projective, and let $B=\End_{A,-}(M)^\opp$.  Then the right pseudotrace ${}^\psi\Tr\in\SLF(B)$ is non-degenerate.
\end{pp}

\begin{proof}
By Prop. \ref{lb2}, $M$ can be viewed as a direct summand of $\bigoplus_{i=1}^n Ae_i$ where each $e_i\in A$ is an idempotent. Let $e\in A$ be an idempotent such that $e\geq e_i$ for all $i$. Then $M$ is a direct summand of $(Ae)^{\oplus n}$. By Prop. \ref{lb6}, we have $\End_{A,-}(Ae)^\opp=eAe$, and hence $\End_{A,-}((Ae)^{\oplus n})^\opp= eAe\otimes\Cbb^{n\times n}$. By Prop. \ref{lb30}, there is an idempotent $p\in eAe\otimes\Cbb^{n\times n}$ such that $M=(Ae)^{\oplus n}p$. By Lem. \ref{lb60}, $\psi|_{eAe}$ is non-degenerate, and hence $\psi|_{eAe}\otimes\trc:eAe\otimes\Cbb^{n\times n}\rightarrow\Cbb$ is non-degenerate. By Lem. \ref{lb60} again, the restriction of $\psi|_{eAe}\otimes\trc$ to $p(eAe\otimes\Cbb^{n\times n})p$ (which is $B$ due to Prop. \ref{lb30}) is non-degenerate. But this restriction is exactly ${}^\psi\Tr$ due to Exp. \ref{lb47} and Prop. \ref{lb40}.
\end{proof}

\begin{thm}\label{lb64}
Assume that $A$ is strongly AUF. Then in Thm. \ref{lb44}, for any $\psi\in\SLF(A)$, the non-degeneracy of $\psi$ and of ${^\psi}\Tr$ are equivalent.
\end{thm}

\begin{proof}
We use the notation in the proof of Thm. \ref{lb44}. From that proof, we know ${^\psi}\Tr=(\psi|_{eAe}\otimes\trc)|_B$. By Lem. \ref{lb60}, $\psi$ is non-degenerate iff $\psi|_{eAe}$ is so, and $\psi|_{eAe}\otimes\trc$ is non-degenerate iff $(\psi|_{eAe}\otimes\trc)|_B$ is so. The equivalence of the non-degeneracy of $\psi|_{eAe}$ and of $\psi|_{eAe}\otimes\trc$ is obvious. The proof is finished. 
\end{proof}

\section{Classification of strongly AUF algebras}\label{lb63}

In this section, we fix an AUF algebra $A$. 

\begin{df}
For each left $A$-module $M$, let $M^*$ be the space of linear functionals, which has a right $A$-module structure defined by
\begin{align*}
(\phi a)(m)=\phi(am)\qquad \text{for all }a\in A,m\in M
\end{align*}
We define the \textbf{quasicoherent dual}
\begin{align*}
M^\vee=&\{\phi\in M^*:\phi\in\phi\cdot A\}\\
=&\{\phi\in M^*:\text{there exists an idempotent $e\in A$ such that $\phi=\phi e$}\}
\end{align*}
By Def. \ref{lb18}, $M^\vee$ is the largest right $A$-submodule of $M^*$ that is quasicoherent.
\end{df}

\begin{rem}\label{lb57}
Let $M\in\QC(A)$. Let $(e_i)_{i\in\fk I}$ be as in Def. \ref{lb4}. Then, as vector spaces, we clearly have
\begin{align*}
M=\bigoplus_{i\in\fk I}e_iM\qquad M^*=\prod_{i\in\fk I}(e_iM)^*
\end{align*}
It follows easily that
\begin{align*}
M^\vee=\bigoplus_{i\in\fk I}(e_iM)^*
\end{align*}
\end{rem}

\begin{df}
For each $M\in\QC(A)$, we let
\begin{align*}
\End^0(M)=M\otimes_\Cbb M^\vee
\end{align*}
viewed as a subalgebra of $\End(M)$.\footnote{That is, for each $\xi\in M,\phi\in M^\vee$, the operator $\xi\otimes\phi$ sends each $\eta\in M$ to $\phi(\eta)\cdot\xi$.} If $B$ is an algebra, and $M$ has a right $B$-module structure commuting with the left action of $A$, we let
\begin{align}
\End^0_{-,B}(M)=\{T\in\End^0(M):(T\xi)b=T(\xi b)\text{ for all }\xi\in M,b\in B\}
\end{align}
\end{df}

\begin{rem}\label{lb58}
Let $M\in\Coh(A)$. By Rem. \ref{lb56} we have $\dim e_iM<+\infty$. It follows from Rem. \ref{lb57} that
\begin{align*}
\End^0(M)=\{T\in\End(M):Te_i=0\text{ for all but finitely many }i\in\fk I\}
\end{align*}
\end{rem}

\begin{pp}\label{lb51}
Choose $M\in \Coh(A)$, and let $B=\End_{A,-}(M)^\opp$. Then for each generating idempotent $p\in B$, we have a linear isomorphism
\begin{gather}\label{eq11}
\begin{gathered}
	\End_{-,B}^0(M)\xlongrightarrow{\simeq} \End_{-,pBp}^0(Mp)\qquad S\mapsto S\big|_{Mp}
\end{gathered}
\end{gather}
\end{pp}

\begin{proof}
Step 1. Let $\wht B=B^\opp=\End_{A,-}(M)$, and let $\wht p\in\wht B$ be the opposite element of $p$. Then $M$ has a left $\wht B$-module structure commuting with the left action of $A$, and $R_p$ is the left multiplication by $\wht p$.

For each $S\in\End^0_{-,B}(M)$, note that $S|_{Mp}=S|_{\wht pM}$ maps $\wht pM$ into $\wht pM$, because $S\wht p\xi=\wht pS\xi\in\wht pM$ for each $\xi\in M$. It is clear that $S|_{Mp}$ commutes with the action of $\wht p\wht B\wht p$. It can be checked from Rem. \ref{lb58} that $S|_{Mp}$ belongs to $\End^0(M)$. This proves that $S|_{Mp}$ belongs to $\End^0_{-,pBp}(Mp)$. We have thus proved that the linear map \eqref{eq11} is well-defined.\\[-1ex]

Step 2. Let us prove the surjectivity of \eqref{eq11}. By Prop. \ref{lb29}, $B$ is finite-dimensional. Therefore, we have an orthogonal primitive decomposition $1_{\wht B}-\wht p=f_1+\cdots +f_n$ in $\wht B$. In this case, we have 
\begin{align*}
M=\wht pM\oplus f_1 M \oplus \cdots \oplus f_n M
\end{align*}
By Prop. \ref{lb16}, for each $1\leq i\leq n$, $f_i$ is equivalent to a sub-idempotent $q_i$ of $\wht p$, i.e., there exist $u_i\in f_i\wht B q_i$ and $v_i\in q_i\wht B f_i$ such that $u_iv_i=f_i$ and $v_iu_i=q_i\leq \wht p$ (where $q_i\in\wht B$ is an idempotent).

Now, we choose $T\in\End_{-,pBp}^0(Mp)=\End_{-,pBp}^0(\wht pM)$. Define a linear map 
	\begin{gather}\label{eq12}
	\begin{gathered}
		S:M\rightarrow M\qquad
		\xi\mapsto T(\wht p\xi)+\sum_{i=1}^n u_i T(v_i\xi)
	\end{gathered}
\end{gather}
By Rem. \ref{lb58}, we have $S\in\End^0(M)$. We claim that $S$ commutes with the action of $\wht B$ (and hence $S\in \End_{-,B}^0(M)$). If this is proved, then since $T$ clearly equals $S|_{Mp}=S|_{\wht pM}$ (because $v_i\wht p=0$, see below), the proof of the surjectivity of \eqref{eq11} is complete.

Note that since $\wht p,f_1,\dots,f_n$ are mutually orthogonal, we have
\begin{gather*}
u_iu_j=0\qquad v_iv_j=0\qquad\text{for all }i,j\\
v_ju_i=0\qquad\text{for all } i\neq j\\
v_i\wht p=0\qquad \wht pu_i=0\qquad\text{for all }i
\end{gather*}
Using this observation and the fact that $T:\wht pM\rightarrow\wht pM$ commutes with the left action of $\wht p\wht B\wht p$, we compute that for each $j$ and $\xi\in M$,
  \begin{gather*}
	S(v_j\xi)=T(\wht pv_j \xi)+0= T(v_j \xi)\\
	v_j S(\xi)=v_jT(\wht p\xi)+v_ju_j T(v_j\xi)\xlongequal{v_ju_j=q_j\in \wht p\wht B\wht p}0+T(q_jv_j\xi)=T(v_j \xi)
  \end{gather*}
and hence $S(v_j\xi)=v_j S(\xi)$; similarly,
  \begin{gather*}
	S(u_j \xi)=T(\wht pu_j \xi)+u_jT(v_ju_j\xi)\xlongequal{v_ju_j=q_j\in \wht p\wht B\wht p}  0+u_jq_jT(\wht p\xi)=u_j T(\wht p\xi)\\
	u_j S(\xi)=u_j T(\wht p\xi)+0=u_j T(\wht p\xi)
  \end{gather*}
and hence $S(u_j\xi)=u_j S(\xi)$. Moreover, for each $b\in\wht B$ we have
\begin{gather*}
S(\wht pb\wht p\xi)=T(\wht pb\wht p\xi)+0=\wht pb\wht pT(\wht p\xi)\\
\wht pb\wht pS(\xi)=\wht pb\wht pT(\wht p\xi)+0=\wht pb\wht pT(\wht p\xi)
\end{gather*}
and hence $S(\wht pb\wht p\xi)=\wht pb\wht pS(\xi)$. This proves that $S$ commutes with the left action of $\wht B$, since $\wht B$ is generated by $\{u_i,v_i:1\leq i\leq n\}$ and $\wht p\wht B\wht p$---to see this, note that for each $b\in\wht B$, by setting $f_0=u_0=v_0=\wht p$, we have
\begin{align*}
b=\sum_{i,j=0}^n f_ibf_j=\sum_{i,j=0}^nu_ib_{i,j}v_j
\end{align*}
where each $b_{i,j}:=v_ibu_j$ commutes with the left actions of $A$ and satisfies $b_{i,j}=\wht pb_{i,j}\wht p$, and hence belongs to $\wht p\wht B\wht p$. \\[-1ex]

Step 3. If $S\in\End^0_{-,B}(M)$ and $S|_{\wht pM}=0$, then for each $\xi\in M$, we have
\begin{align*}
S(\xi)=S(\wht p\xi)+\sum_{i=1}^n S(f_i\xi)=S(\wht p\xi)+\sum_{i=1}^n u_iS(v_i\xi)
\end{align*}
where $\wht p\xi,v_j\xi\in\wht p M$. Therefore $S=0$. This proves that \eqref{eq11} is injective.
\end{proof}

\begin{lm}\label{lb54}
Suppose that $e\in A$ is a generating idempotent. Then we have a linear isomorphism
\begin{align}\label{eq15}
A\xlongrightarrow{\simeq}\End_{-,eAe}^0(Ae)
\end{align}
sending each $a\in A$ to the left multiplication by $a$.
\end{lm}

\begin{proof}
It is obvious that the left action on $Ae$ by $a\in A$ belongs to $\End^0_{-,eAe}(Ae)$. Therefore, the map \eqref{eq15} is well-defined.

Suppose that the left multiplication of $a\in A$ on $Ae$ is zero. Then $aAe=0$. Since $A$ is AUF and hence almost unital, there is an idempotent $p\in A$ such that $a=ap$. Since $e$ is generating, by Cor. \ref{lb32}, $Ae$ is a generator of $\Coh(A)$. Therefore, $Ap$ is a quotient module of $(Ae)^{\oplus n}$ for some $n\in\Zbb_+$. Thus $aAp$ is a quotient space of $(aAe)^{\oplus n}$, and hence $aAp=0$. This proves $ap=0$, and hence $a=0$. We have thus proved that \eqref{eq15} is injective.

Choose $T\in \End_{-,eAe}^0(Ae)$. Since $T\in\End^0(Ae)$, by Rem. \ref{lb57}, there is an idempotent $f\in A$ such that $T=fTf$. It follows that $fTf|_{fAe}$ belongs to $\End_{-,eAe}(fAe)$. Since $A$ is AUF, we may enlarge $f$ so that $e\leq f$ also holds. We claim that $\End_{-,eAe}(fAe)$ consists of the left multiplications by elements of $fAf$. If this is true, then $T|_{fAe}=fTf|_{fAe}$ is the left multiplication by $faf$ for some $a\in A$. It follows that for any $b\in A$, we have $Tbe=Tfbe=fafbe$, and hence $T$ is the left multiplication by $faf$ on $Ae$, finishing the proof that \eqref{eq15} is surjective. 

By Cor. \ref{lb59}, the idempotent $e\in fAf$ is generating in $fAf$. Applying Prop. \ref{lb51} to the finite-dimensional unital algebra $fAf$ and its (finite-dimensional) coherent left module $fAf$, we see that $\End_{-,eAe}(fAe)=fAf|_{fAe}$. This proves the claim.
\end{proof}

\begin{thm}\label{lb61}
Suppose that $A$ is strongly AUF, and let $G$ be a projective generator of $\Coh(A)$ (which exists due to Prop. \ref{lb50}). Set $B=\End_{A,-}(G)^\opp$. Regard $G$ as an $A$-$B$ bimodule. Then we have a linear isomorphism
\begin{align}\label{eq16}
A\xlongrightarrow{\simeq}\End^0_{-,B}(G)
\end{align}
sending each $a\in A$ to the left multiplication of $a$ on $G$.
\end{thm}

\begin{proof}
By Cor. \ref{lb33} and Prop. \ref{lb50}, $A$ has a generating idempotent $e$. If $G=Ae$, then $\End_{A,-}(G)^\opp=eAe$ due to Prop. \ref{lb6}. Therefore, by Lem. \ref{lb54}, the map \eqref{eq16} is bijective.

If $G=(Ae)^{\oplus n}$ where $n\in\Zbb_+$, one easily checks that $B=eAe\otimes\Cbb^{n\times n}$ where $\Cbb^{n\times n}$ is the matrix algebra of order $n$. The bijectivity of \eqref{eq16} then follows easily.

Finally, let $G$ be any general projective generator. By Cor. \ref{lb43}, we may assume that $G=(Ae)^{\oplus n}p$ where $n\in\Zbb_+$, and $p$ is a generating idempotent of $\wtd B=\End_{A,-}((Ae)^{\oplus n})^\opp\simeq eAe\otimes\Cbb^{n\times n}$. By Prop. \ref{lb30}, we have $B=p\wtd Bp$. Therefore, by Prop. \ref{lb51}, the map 
\begin{align*}
\End^0_{-,\wtd B}((Ae)^{\oplus n})\rightarrow \End^0_{-,B}(G)
\end{align*}
sending each $S$ to $S|_G$ is bijective. By the previous paragraph, the map
\begin{align*}
A\rightarrow \End^0_{-,\wtd B}((Ae)^{\oplus n})
\end{align*}
sending each $a$ to the left multiplication by $a$ is bijective. Therefore, their composition, namely \eqref{eq16}, is bijective.
\end{proof}

\begin{rem}\label{lb62}
In Thm. \ref{lb61}, the right $B$-module $G$ is a \textbf{projective generator} in the category $\ModR(B)$ of right $B$-modules---that is, $G$ is projective in $\ModR(B)$, and any object in $\ModR(B)$ has an epimorphism from a (possibly infinite) direct sum of $G$.
\end{rem}

\begin{proof}
The projectivity of $G$ in $\ModR(B)$ is due to Thm. \ref{lb44} and Rem. \ref{lb20}. Using the notation in the proof of Thm. \ref{lb61}, we may assume $G=(Ae)^{\oplus n}p$ and $B=p(eAe\otimes\Cbb^{n\times n})p$ where $e\in A$ and $p\in eAe\otimes\Cbb^{n\times n}$ are generating idempotents. Since $B$ is unital, $B$ is generating in $\ModR(B)$. Therefore $(eAe\otimes\Cbb^{n\times n})p$ is generating in $\ModR(B)$. Since $(eAe\otimes\Cbb^{n\times n})p$  is a direct sum of $(eAe\otimes\Cbb^{1\times n})=(eAe)^{\oplus n}p=eG$, we conclude that $eG$ is generating in $\ModR(B)$. Therefore $G$ is generating in $\ModR(B)$.
\end{proof}

\begin{thm}\label{lb65}
Let $\mc A$ be an algebra. The following are equivalent.
\begin{enumerate}[label=(\arabic*)]
\item $\mc A$ is strongly AUF.
\item $\mc A$ is isomorphic to $\End^0_{-,B}(M)$ where $B$ is a unital finite-dimensional algebra, $M$ is a projective generator in $\ModR(B)$, the vector space $M$ has a grading
\begin{align*}
M=\bigoplus_{i\in\fk I}M(i)
\end{align*}
where each $M(i)$ is finite-dimensional and is preserved by the right action of $B$, and $\End^0_{-,B}(M)$ is defined by
\begin{align*}
\End^0_{-,B}(M):=\{T\in\End(M):&(Tm)b=T(mb)\text{ for all $m\in M,b\in B$,}\\
&T|_{M(i)}=0\text{ for all but finitely many }i\in\fk I \}
\end{align*}
\end{enumerate}
\end{thm}

\begin{proof}
The direction (1)$\Rightarrow$(2) follows from Thm. \ref{lb61} and Rem. \ref{lb62}. Let us prove the other direction.

Assume that $\mc A=\End^0_{-,B}(M)$ where $\End^0_{-,B}(M)$ is described as in (2). Let $e_i$ be the projection of $M$ onto $M(i)$. Then $e_i$ clearly belongs to $\mc A$, and each $T\in\mc A$ can be written as $T=\sum_{i,j\in\fk I}e_i Te_j$ where $e_iTe_j=0$ for all but finitely many $i,j$. This proves that $\mc A$ is AUF.

Since $M$ is a projective generator in $\ModR(B)$, for each finite subset $I\subset\fk I$, $M_I:=\bigoplus_{i\in I}M(i)$ is projective in $\ModR(B)$ (since it is a direct summand of $M$). Let $1_B=p_1+\cdots+p_n$ be an orthogonal primitive decomposition of $1_B$ in $B$. By Thm. \ref{lb15}, irreducible finite-dimensional right $B$-modules are precisely those that are isomorphic to $p_kB/\rad(p_kB)$ for some $k$. Since $M$ is generating in $\ModR(B)$, it has an epimorphism to $p_kB/\rad(p_kB)$ for each $k$. This epimorphism must restrict to a nonzero morphism (and hence an epimorphism) $M(i_k)\rightarrow p_kB/\rad(p_kB)$. Let $I=\{i_1,\dots,i_n\}$. Then $M_I$ has an epimorphism to each irreducible right $B$-module. It follows from Prop. \ref{lb27} that $M_I$ is a projective generator in the category of finite-dimensional right $B$-modules. 

Let $e_I=\sum_{i\in I}e_i$, which is an idempotent in $\mc A$. We claim that $e_I$ is a generating idempotent in $\mc A$, which will complete the proof that $\mc A$ is strongly AUF.

Let $\eps$ be any primitive idempotent of $\mc A$. Then $\eps M$ is a finite-dimensional right $B$-module, since any element of $\mc A$ has finite range when acting on $M$. Moreover, since $\eps$ is primitive in $\mc A$, the right $B$-module $\eps M$ is indecomposable. Since $\eps M$ is a direct summand of the projective right $B$-module $M$, it follows that $\eps M$ is a finite-dimensional indecomposable projective right $B$-module. Therefore, since $M_I=e_IM$ is a projective generator, similar to the end of the proof of Thm. \ref{lb31}, we conclude that the right $B$-module $\eps M$ is isomorphic to a direct summand of $e_IM$. Thus, by Prop. \ref{lb30}, $\eps$ is equivalent to a sub-idempotent of $e_I$ in $\mc A$. This proves the claim that $e_I$ is generating.
\end{proof}

\footnotesize
	\bibliographystyle{alpha}
%    \bibliography{voa}

\noindent {\small \sc Yau Mathematical Sciences Center, Tsinghua University, Beijing, China.}

\noindent {\textit{E-mail}}: binguimath@gmail.com\qquad bingui@tsinghua.edu.cn\\

\noindent {\small \sc Yau Mathematical Sciences Center and Department of Mathematics, Tsinghua University, Beijing, China.}

\noindent {\textit{E-mail}}: zhanghao1999math@gmail.com \qquad h-zhang21@mails.tsinghua.edu.cn
\end{document}